\documentclass[12pt]{article}
\usepackage{amsmath,amsthm,amssymb}  
\usepackage{graphicx}

\newcommand{\C}{\mbox{\rm \,l\kern-0.52em C}}
\newcommand{\Ce}{\rm \,l\kern-0.35em C}

\newcommand{\Inf}{{\rm Inf}}
\newcommand{\Sup}{{\rm Sup}}

\newtheorem{theorem}{Theorem}[section]

\newtheorem{definition}[theorem]{Definition}
\newtheorem{prop}[theorem]{Proposition}

\newtheorem{convent}[theorem]{Convention}

\newtheorem{lemma}[theorem]{Lemma}
\newtheorem{remark}[theorem]{Remark}

\renewenvironment{proof}{{\bf Proof:}}{\mbox{}\hfill $\Box$}

\theoremstyle{definition}
\date{}

\title{Finitely summable $\gamma$-elements for word-hyperbolic groups}

\author{Jean-Marie Cabrera, Michael Puschnigg}

\begin{document} 

\maketitle

\centerline{\it Dedicated to Gennadi Kasparov on the occasion of his 70th birthday}

{\abstract{We present two explicit combinatorial constructions of finitely summable reduced "Gamma"-elements $\gamma_r\,\in\,KK(C^*_r(\Gamma),{\mathbb C})$ for any word-hyperbolic group $(\Gamma,S)$ and obtain summability bounds for them in terms of the cardinality of the generating set $S\subset\Gamma$ and the hyperbolicity constant of the associated Cayley graph.}}

\section{Introduction}

Hyperbolic groups form a large and quite rich class of finitely generated discrete groups. They are characterized by the fact that geodesic triangles in the Cayley-graph ${\mathcal G}(\Gamma,S)$ of a hyperbolic group $\Gamma$ with respect to a finite symmetric set of generators $S$ are $\delta$-thin, where the hyperbolicity constant $\delta$ depends on $\Gamma$ and $S$. Several conjectures which are completely open in general have been established for hyperbolic groups: Kasparov's strong version of the Novikov Conjecture \cite{CM}, \cite{KS}, the Baum-Connes Conjecture \cite{La1},\,\cite{MY}, and even the Baum-Connes Conjecture with coefficients \cite{La2} are now known to hold for such groups.\\
\\
In most of these cases the key ingredient of the proof is the construction of a well behaved Fredholm module 
$$
\mathcal{E}\,=\,({\mathcal H},\,\pi,\,F)
\eqno(1.1)
$$
over the hyperbolic group $\Gamma$. This module is supposed to represent the "Gamma"-element
$$
\gamma \,\in\,KK_\Gamma({\mathbb C},{\mathbb C})
\eqno(1.2)
$$
in Kasparov's equivariant bivariant $K$-theory. Recall that a Fredholm module over $\Gamma$ is given by an even unitary representation of $\Gamma$ on the ${\mathbb Z}/2{\mathbb Z}$-graded separable Hilbert space $\mathcal H$, and an odd bounded linear operator $F$ on 
$\mathcal H$, whose image in the Calkin algebra ${\mathcal L}({\mathcal H})/{\mathcal K}({\mathcal H})$ is selfadjoint, unitary, and commutes with $\pi(\Gamma)$. The homotopy classes of such bimodules form then Kasparov's bivariant $K$-group.

\newpage

None of these modules are known to possess the property of finite summability, 
which demands the previous conditions to hold already modulo some ideal of Schatten-operators 
$\ell^p({\mathcal H})\subset {\mathcal K}({\mathcal H})$. 
Finitely summable Fredholm modules possess nice regularity properties. In particular, the Chern character of a finitely summable Fredholm module in cyclic cohomology can be given by a simple formula  \cite{Co}. It is therefore an interesting question whether a given $K$-homology class 
can be realized by a finitely summable Fredholm module. \\
\\
In this paper we propose two explicit constructions of finitely summable Fredholm modules over $C^*_r(\Gamma)$ representing the class of Kasparov's reduced "Gamma"-element
$$
\gamma_r\,\in\,KK(C^*_r(\Gamma),{\mathbb C}).
\eqno(1.3)
$$
This element maps to the "Gamma"-element (1.2) under the pull-back along the canonical morphism $C^*(\Gamma)\to C^*_r(\Gamma).$ Our result solves a problem posed in \cite{EN}.\\
\\
For discrete isometry groups of hyperbolic space such Fredholm modules were exhibited by Connes \cite{Co}. Their construction is based on the existence and uniqueness of geodesic segments and the fact that angles in a geodesic triangle in hyperbolic space decay exponentially with the distance from the opposite side.
\\
None of these properties is inherited by general $\delta$-hyperbolic metric spaces but Mineyev's ideas about homological bicombings \cite{M} 
allow to find appropriate substitutes for them. Instead of geodesic segments joining a base vertex $x$ and an auxiliary vertex $y$ in the Cayley graph of 
a $\delta$-hyperbolic group we consider the family $\Omega_{x,y}$ of all regular sequences (beginning at $x$ and ending at $y$). The elements of such a sequence lie uniformly close to the locus $geod\{x,y\}$ of all geodesic segments joining $x$ and $y$ and the distance between two consecutive elements
is almost fixed and large compared to $\delta$. By Mineyev's work $\Omega_{x,y}$ carries a natural probability measure. The set $\Omega^r_{x,y}$ of tails of regular sequences of length at most $r$ inherits a natural probability measure and the mass of the symmetric difference of $\Omega^r_{x,y}$ and $\Omega^r_{x',y}$ decays exponentially with the distance between $geod\{x,x'\}$ and the $r$-ball centered at $y$. This replaces the exponential decay condition 
for triangles in hyperbolic space mentioned above.\\
\\
We use regular sequences as a tool to modify two well known constructions of reduced "Gamma"-elements for word-hyperbolic groups. Both of them are given by operators on the closed subspace  
$$
{\mathcal H}^R(\Gamma)\,=\,\underset{n=0}{\overset{\infty}{\bigoplus}}\,
\langle e_{x_0}\wedge e_{x_1}\wedge\ldots\wedge e_{x_n},\,\{x_0,\ldots,x_n\}\subset\Gamma,\,d(x_i,x_j)\leq R,\,0\leq i,j\leq n\rangle
\eqno(1.4)
$$
of the Hilbert space $\Lambda^*(\ell^2(\Gamma))$ for $R>0$ sufficiently large. The Hilbert space ${\mathcal H}^R(\Gamma)$ is a 
completion of the alternating Rips chain complex, which provides a finite free resolution of the constant $\Gamma$-module $\mathbb C$.\\
\\
In his proof of the Baum-Connes Conjecture with coefficients for word-hyperbolic groups \cite{La2}
Lafforgue gave a detailed analysis of $K$-cycles of the form
$$
{\mathcal E}_\gamma\,=\,
\left({\mathcal H}^R(\Gamma),\,\pi_{reg},\,e^{t\cdot d^{Laff}_x}\circ(\partial\,+\,h^{Laff}_x)\circ e^{-t\cdot d^{Laff}_x}\right),
\eqno(1.5)
$$
representing the reduced "Gamma"-element for $R$ and $t$ sufficiently large. Here $\partial$ is the simplicial differential of the Rips chain complex and $h^{Laff}_x$ is a simplicial homotopy operator of square zero contracting the Rips complex to the base point $x$. Such an operator is given by a filling procedure for cycles. We use an alternative algorithm given by projecting the given cycle "orthogonally" onto a nearby regular sequence and filling its image 
inside that sequence. This only requires a good filling of cycles in the metric space $\mathbb N$ and a classical homotopy formula for Rips complexes to correct the committed error. The natural measure on the set of regular sequences permits to average the obtained fillings and thus to get rid of their dependence on the choices made. The diagonal operator $d^{Laff}_x$ in (1.5) multiplies a basis vector corresponding to a Rips simplex with a "twisted" distance to the origin. This twisted distance is quasiisometric to the word metric but satisfies in addition the decay condition
$$
\underset{r\to\infty}{\lim}\,\underset{\underset{d(x,y)\geq r}{d(x,x')=d(y,y')=1}}{\sup}\,
\vert d(x,y)-d(x',y)-d(x,y')+d(x',y')\vert\,=\,0
\eqno(1.6)
$$
We replace Lafforgue's metric by the metric of Mineyev-Yu, whose construction is also based on Mineyev's bicombing. The modified Lafforgue-bimodules obtained in this way still represent the reduced "Gamma"-element and turn out to be in addition finitely summable due to the behavior  
of regular sequences.\\
\\
The first construction of Fredholm modules representing the "Gamma"-element of general hyperbolic groups goes actually back to Kasparov and Skandalis \cite{KS}. They use the same Hilbert space, but their operator is local and given by Clifford multiplication 
$$
e_{x_0}\wedge e_{x_1}\wedge\ldots\wedge e_{x_n}\,\mapsto\,cl(\zeta_{\{x_0,\ldots,x_n\}})(e_{x_0}\wedge e_{x_1}\wedge\ldots\wedge e_{x_n})
\eqno(1.7)
$$
with suitable vectors $\zeta_{\{x_0,\ldots,x_n\}}\in{\mathbb C}\Gamma$. We replace these vectors by an average (using Mineyev's measure) of the appropriate tails of all regular sequences starting at the base vertex end ending in $\{x_0,\ldots,x_n\}$. The modified Fredholm modules represent again the reduced "Gamma"-element but are in addition $p$-summable for 
$$
p>20\delta\cdot\log(1+\vert S\vert)\cdot(1+\vert S\vert)^{2\delta}
\eqno(1.8)
$$
In fact much better (and presumably optimal) bounds are known in certain cases. Emerson and Nica \cite{EN} give, by using the Gysin-sequence in $K$-homology relating the trivial and the boundary action of $\Gamma$, a very elegant abstract existence proof of finitely summable "Gamma"-elements over the maximal group $C^*$-algebra of a word-hyperbolic group of Euler-Poincar\'e characteristic zero. They obtain in this case the geometric summability bound $p\,\geq\,Max(visdim(\partial\Gamma),2)$, where $visdim(\Gamma)$ denotes the Hausdorff dimension of the boundary of $\Gamma$ with respect to a visual metric.\\
\\
It should be noted that the finite summability of "Gamma"-elements is a rather exceptional phenomenon. Higher rank lattices for example behave very differently in this respect: no nontrivial $K$-homology class of the reduced group $C^*$-algebra of a higher rank lattice can be finitely summable over the group algebra \cite{Pu}.
\\
\\
Finally we express our hope that the construction of "nice" Fredholm modules representing "Gamma"-elements might lead in the future to simplified proofs of the Baum-Connes Conjecture with coefficients for hyperbolic groups following the lines of Lafforgue's monumental paper \cite{La2}.
This was our key motivation and explains why we put the main emphasis on Lafforgue's bimodule.
\\
\\
This work is based on the first authors thesis supervised by the second author.\\
\\
Georges Skandalis observed that a short alternative proof of the existence of finitely summable "Gamma"-elements for hyperbolic groups might be obtained by 
applying Mineyev's ideas to the Fredholm modules used by him and Kasparov in their work on the Novikov conjecture \cite{KS}. We thank him heartily for his suggestion and for enlightening discussions about his work with Kasparov.

\tableofcontents

\section{Rips complexes of hyperbolic groups}

\subsection{Hyperbolic spaces}

For a subset $Y\subset X$ of a metric space $(X,d)$ and $R\geq 0$ we note
$$
B(Y,R)\,=\,\{x\in X, d(x,Y)\leq R\}\,=\,\{x\in X,\underset{y\in Y}{\Inf}\,d(x,y)\leq R\}
\eqno(2.1)
$$
and
$$
geod(Y)\,=\,\{x\in X,\,\exists y,y'\in Y:\,d(y,x)+d(x,y')=d(y,y')\}.
\eqno(2.2)
$$ 

Recall that a metric space $(X,d)$ is {\bf geodesic} if any pair of points $x,y\in X$ can be joined by a {\bf geodesic segment}, i.e. if there exists an isometric map
$\gamma:I\to X$ from a bounded closed interval to $X$ such that $\gamma(\partial I)=\{x,\,y\}$.
A geodesic triangle with vertices $x,y,z\in X$ is given by 
three geodesic segments $[x,y],\,[y,z],\,[x,z]$, joining the denoted endpoints. 

\begin{definition} (Gromov)  \cite{Gr},\cite{GH}\\
A geodesic metric space $(X,d)$ is {\bf $\delta$-hyperbolic} (for some 
$\delta\geq 0$) if each edge in a geodesic triangle is contained in the tubular $\delta$-neighbourhood of the union of the two other edges:
$$
[x,z]\,\subset\,B([x,y]\cup [y,z],\delta)
\eqno(2.3)
$$
\end{definition}

The {\bf Gromov product} of three points $x,y,z$ in a metric space is defined as
$$
(x\vert y)_z\,=\,\frac12(d(x,z)+d(y,z)-d(x,y))
\eqno(2.4)
$$
It is quite useful in $\delta$-hyperbolic metric spaces because of the estimate
$$
(x\vert y)_z\,\leq\,d(z,geod\{x,y\})\,\leq\,(x\vert y)_z+2\delta,\,\,\,\forall x,y,z\in X.
\eqno(2.5)
$$

\subsection{Hyperbolic groups  \cite{Gr},\cite{GH}} 

Let $(\Gamma,S)$ be a finitely generated group with associated word length function $\ell_S$ and word metric $d_S$. The corresponding Cayley graph ${\mathcal G}(\Gamma,S)$ with vertices 
${\mathcal G}(\Gamma,S)_0=\Gamma$ and edges ${\mathcal G}(\Gamma,S)_1=\Gamma\times S,\,
\partial_0(g,s)=g,\,\partial_1(g,s)=gs$  is a proper geodesic metric space on which $\Gamma$ acts properly, isometrically and cocompactly by left translation. The group $\Gamma$ is called {\bf hyperbolic} if its Cayley graph with respect to some (and thus to every) finite, symmetric set of generators is hyperbolic in the sense of 2.1. (The constant $\delta$ depends of course on the choice of $S$.) By abuse of language we call the pair $(\Gamma,S)$ a {\bf $\delta$-hyperbolic group} if ${\mathcal G}(\Gamma,S)$ is a $\delta$-hyperbolic space. We suppose in the sequel that $\delta$ is a strictly positive integer.

\begin{convent}
We fix for every element $g\in G$ a word $w(g)\in S^{\ell(g)}$ of minimal length representing it.
Such a choice defines for any $x\in\Gamma$ a geodesic path $\overline{xy}$ in ${\mathcal G}(\Gamma,S)$ joining $x$ and $y=xg$. Its edges are labeled by the consecutive letters of $w(g)$. This construction is equivariant i.e. it commutes with left-multiplication by $\Gamma$. For $0\leq t\leq d(x,y)$ we denote by $\overline{xy}(t)$ the point of $\overline{xy}$ lying at distance $t$ from $x$.
\end{convent}

\subsection{Bar complexes and Rips complexes}

We recall a few facts about standard resolutions of modules over group rings.

\begin{definition}
The {\bf Bar-complex} $\Delta_\bullet(X)$ of a set $X$ is the simplicial set 
with $n$-simplices $\Delta_n(X)\,=\,X^{n+1}$, face maps
$$
\partial_i([x_0,\ldots,x_n])\,=\,[x_0,\ldots,x_{i-1},x_{i+1},\ldots,x_n],
\eqno(2.6)
$$
and degeneracy maps
$$
s_j([x_0,\ldots,x_n])=[x_0,\ldots,x_j,x_j,\ldots,x_n]\,\,\text{for}\,\,0\leq i,j\leq n\in{\mathbb N}.
\eqno(2.7)
$$
The {\bf support} of a Bar-simplex is
$
Supp([x_0,\ldots,x_n])\,=\,\{x_0,\ldots,x_n\}\subset X.
$
\end{definition}

\begin{definition}
Let $(X,d)$ be a metric space and let $R\geq 0$. The {\bf Rips-complex} $\Delta_\bullet^R(X)$
of $(X,d)$ is the simplicial subcomplex of the Bar-complex $\Delta_\bullet(X)$ given by the Bar-simplices of diameter at most $R:$
$$
\Delta_n^R(X)\,=\,\{[x_0,\ldots,x_n]\in X^{n+1},\,d(x_i,x_j)\leq R,\,0\leq i,j\leq n\}.
\eqno(2.8)
$$
\end{definition}
Every map of sets $f:X\to Y$ gives rise to a simplicial map 
$$
f_\bullet:\,\Delta_\bullet(X)\to\Delta_\bullet(Y),\,[x_0,\ldots,x_n]\mapsto [f(x_0),\ldots,f(x_n)].
\eqno(2.9)
$$
In particular, every group action on the set $X$ gives rise to a simplicial action on the Bar-complex $\Delta_\bullet(X)$ and every isometric group action on a metric space $(X,d)$ gives rise 
to a simplicial action on the Rips-complexes $\Delta_\bullet^R(X)$ for any $R>0$.\\

\begin{definition}
The {\bf Bar chain complex} $C_*(X,{\mathbb Z})$ of a set $X$ is given by the free abelian group with basis $\Delta_*(X)$ modulo the subspace spanned by degenerate simplices. Its differentials are given by the alternating sum of the linear operators induced by the face maps. The {\bf Rips chain complexes} $C_*^R(X,{\mathbb Z})$ of a metric space are defined similarly. They are subcomplexes of the Bar chain complex.
\end{definition}

The {\bf support} of a Bar-chain is the union of the support of the simplices occuring in it with nonzero multiplicity.\\
\\
The augmentation map $C_0(X,{\mathbb Z}) \to{\mathbb Z}$ of the Bar-(resp. Rips-)complex sends any zero simplex to 1. The augmented Bar-complex is contractible (but there is no natural contraction). If $x\in X$ is a base point, then 
$$
\begin{array}{cccc}
s_x: & C_*(X,{\mathbb Z}) & \to & C_{*+1}(X,{\mathbb Z}) \\
& & & \\
& [x_0,\ldots,x_n] & \mapsto & [x,x_0,\ldots,x_n] \\
\end{array}
\eqno(2.10)
$$
is a contracting homotopy of the augmented Bar-complex:
$$
Id\,=\,\partial\circ s_x\,+\,s_x\circ\partial.
$$
In particular, the homology of the Bar-complex is of rank one and concentrated in degree zero. 
The augmentation map identifies it canonically with $\mathbb Z$.

\begin{prop} (Gromov)  \cite{Gr} pp.101,96.
Let $(\Gamma,S)$ be a $\delta$-hyperbolic group. Then the augmented Rips chain complex
 $C_*^R(\Gamma,{\mathbb Z})$ is contractible for $R\geq 4\delta$.
\end{prop}

The following two lemmata are easily verified by direct calculation.
\begin{lemma}
Let $\varphi_*,\psi_*:C_*(X,{\mathbb Z})\to C_*(Y,{\mathbb Z})$ be chain maps of Bar complexes
which induce the identity in homology (i.e. which are compatible with the augmentations).
Then the linear operator which vanishes in degree -1 and equals 
$$
\begin{array}{cccc}
h(\varphi,\psi): & C_*(X,{\mathbb Z}) & \to & C_{*+1}(Y,{\mathbb Z}) \\
& [x_0,\ldots, x_n] & \mapsto & \underset{i=0}{\overset{n}{\sum}}\,(-1)^i\,[\varphi_i(x_0,\ldots,x_i),
\psi_{n-i}(x_i,\ldots,x_n)] \\
\end{array}
\eqno(2.11)
$$
in nonnegative degrees defines a natural chain homotopy between $\varphi$ and $\psi$:
$$
\psi_*-\varphi_*\,=\,\partial\circ h(\varphi,\psi)\,+\,h(\varphi,\psi)\circ\partial.
\eqno(2.12)
$$
In particular, if $G$ is a group acting on $X$ and $Y$, and if $\varphi_*$ and $\psi_*$ are $G$-equivariant, then $h(\varphi,\psi)$ is $G$-equivariant as well.
\end{lemma}

\begin{lemma}
The antisymmetrization operator
$$
\begin{array}{cccc}
\pi_{alt}: & C_*(X,{\mathbb Q}) & \to & C_*(X,{\mathbb Q}) \\
& & & \\
& [x_0,\ldots,x_n] & \mapsto & \frac{1}{(n+1)!}\underset{\sigma\in\Sigma_{n+1}}{\sum}\,(-1)^{\epsilon(\sigma)}\,[x_{\sigma(0)},\ldots,x_{\sigma(n)}] \\
\end{array}
\eqno(2.13)
$$
is a chain map which preserves the Rips subcomplexes and equals the identity in degree zero. In particular it is naturally chain homotopic to the identity by the previous lemma.
\end{lemma}

\subsection{Filling cycles near geodesic segments}

How to find a contracting chain homotopy of the Rips complex ?
In degree zero one would expect an operator which attaches to a given vertex a geodesic segment joining it to a fixed origin or base vertex. Following an idea of Mineyev \cite{M} we use instead of a single geodesic a weighted average of regular sequences of equidistant vertices  close to such geodesic segments. The advantage of this procedure is that it depends in a strictly controlled way on the choice of the origin. In fact, the difference of the weights of a given vertex of a regular sequence with respect to two different origins decays exponentially with the distance from the origins. This exponential decay property will be responsible for the finite summability of the Fredholm modules we are going to construct. To obtain the contracting homotopy in higher degrees it suffices then to fill cycles supported near a regular sequence of vertices in the Cayley-graph.
Taking up an idea of Bader, Furman and Sauer \cite{BFS}, we project a given cycle "orthogonally" 
onto the regular sequence and obtain a cycle supported in this sequence. The latter may be viewed as a Rips cycle of positive degree in the metric space $\mathbb N$ and can be filled canonically. The committed error depends on the distance of the initial cycle from the regular sequence and will be corrected using the classical homotopy formula for maps of Bar-complexes.

\subsubsection{Filling cycles in segments}

\begin{lemma}
There exists a contracting chain homotopy
$$
\sigma_*:\,C_*({\mathbb N},{\mathbb Z})\to C_{*+1}({\mathbb N},{\mathbb Z}),\,\,\,*\geq -1,
\eqno(2.14)
$$
of the augmented Bar-complex of (the metric space) $\mathbb N$ such that
$$
Supp(\sigma_*(\alpha))\,\subset\,geod(Supp\,\alpha)
\eqno(2.15)
$$
and
$$
\parallel \sigma_*(\alpha)\parallel_1\,\leq\,diam(Supp(\alpha))
\eqno(2.16)
$$
for all Bar-simplices $\alpha\in\Delta_*({\mathbb N})$ of dimension 
$*\geq 1$. In fact
$$
\sigma(C_*^R({\mathbb N},{\mathbb Z}))\,\subset\,C_{*+1}^R({\mathbb N},{\mathbb Z})
\eqno(2.17)
$$
for $*\geq -1$ and every $R\in{\mathbb N}^*$.
The same assertions hold for the subcomplexes $C_*^R(\{0,\ldots,m\},{\mathbb Z}),\,m\in{\mathbb N}$.
\end{lemma}

\begin{proof}
Put $\sigma_{-1}(1)=[0]$ and $\sigma_0([n])=[0,1]+\ldots+[n-1,n]$ and define inductively 
$$
\sigma_n([x_0,\ldots,x_n])\,=\,s_{x_0}\circ(Id-\sigma_{n-1}\circ\partial)([x_0,\ldots,x_n]),\,\,n\geq 1.
$$
(See (2.10) for the definition of $s_{x_0}$.) One verifies easily that $\sigma\circ\partial+\partial\circ\sigma=id$ and that
$$
\sigma_n([x_0,\ldots,x_n])\,=
\begin{cases}
(-1)^{n}\underset{k=x_{n-1}+1}{\overset{x_n}{\sum}}\,[x_0,\ldots,x_{n-1},k-1,k], & x_n > x_{n-1}, \\
 & \\
(-1)^{n-1}\underset{k=x_{n}+1}{\overset{x_{n-1}}{\sum}}\,[x_0,\ldots,x_{n-1},k-1,k], & x_n<x_{n-1} \\
\end{cases}
\eqno(2.18)
$$
modulo degenerate simplices.
In particular 
$$Supp(\sigma_n([x_0,\ldots,x_n]))=\{x_0,\ldots,x_{n-1}\}\cup geod\{x_{n-1},x_n\}\subset geod\{x_0,\ldots,x_n\}$$ and $\parallel\sigma_*([x_0,\ldots,x_n])\parallel=\vert x_n-x_{n-1}\vert\leq diam\{x_0,\ldots,x_n\}$ 
so that 
$Supp(\sigma_*(\alpha))\subset (geod(Supp(\alpha))$ and
$\parallel\sigma_*(\alpha)\parallel\leq diam(Supp(\alpha))$ 
for all $\alpha\in\Delta_*({\mathbb N}),\,*>0$ as desired.
\end{proof}

\subsubsection{Regular sequences}

\begin{remark}
In the sequel various constants will come up in our statements. These are viewed as functions of various parameters and are monotone increasing as functions of the numerical parameters among them. In particular, they depend exclusively on the mentioned parameters, which will be the main point of interest. 
\end{remark}

We recall a construction of Mineyev, which will play a crucial role in the sequel.
Let us begin with a few motivating remarks. 
If $x,y$ are two points in hyperbolic space there is a unique point $\varphi_t(x,y)$ situated at distance $t\leq d(x,y)$ from $y$ and lying on the unique geodesic segment joining $x$ and $y$. Moreover, if $x'$ is a further point, the distance between $\varphi_t(x,y)$ and $\varphi_t(x',y)$ decays exponentially with the distance of these points from the geodesic segment joining $x$ and $x'$.\\
This no longer holds for $\delta$-hyperbolic spaces, but Mineyev constructs 
for any two vertices $x,y$ in the Cayley graph of a $\delta$-hyperbolic group and every integer $k\leq\frac{d(x,y)}{10\delta}$ a 0-chain (in fact a convex combination) $\varphi_k(x,y)$ of vertices at distance $10k\delta$ from $y$ and located $\delta$-close to $geod\{x,y\}$, which has similar properties as the points $\varphi_t(x,y)$ considered above. \\
The scale $10\delta$ used by Mineyev could be replaced by any other scale strictly larger than $\delta$. On scales above $\delta$ the $\delta$-hyperbolicity condition gives a very precise hold on the geometry of the Cayley graph. Mineyev uses clever averaging procedures to get rid of the individual geometry of the Cayley graph on scales below $\delta$. 
In this way he obtains the exponential decay condition (2.22) below.

\begin{prop} (Mineyev )\cite{M}
Let $(\Gamma,S)$ be a $\delta$-hyperbolic group. For each integer $k>0$ there exists a map
$$
\begin{array}{cccc}
\varphi_k: & \Gamma\times\Gamma & \to & C_0(\Gamma,{\mathbb Q}),\\
 & & & \\
 & (x,y) & \mapsto & \underset{z}{\sum}\,c_k^{x,y}(z)[z] \\
\end{array}
\eqno(2.19)
$$
satisfying the following conditions for all $x,x',y,z\in\Gamma$:

\begin{itemize}
\item $\varphi_k(x,y)$ is a convex combination of vertices:
$$
\begin{array}{cc}
c_k^{x,y}(z)\geq 0, & \underset{z}{\sum}\,c_k^{x,y}(z)=1. \\
\end{array}
\eqno(2.20)
$$
\item $\varphi_k(x,y)=[x]$ if $d(x,y)\leq 10k\delta$.
\item $$Supp\,\varphi_k(x,y)\subset\,S(y,10k\delta)\cap B(\overline{xy},\delta)
=\{z\in B(\overline{xy},\delta), d(z,y)=10k\delta\}
\eqno(2.21)$$
 if $d(x,y)> 10k\delta.$
\item $\varphi_k$ is $\Gamma$-equivariant: $\varphi_k(gx,gy)\,=\,g\varphi_k(x,y),\,\forall g\in\Gamma.$
\item There exist constants $C_1(\delta,\vert S\vert)>0,\lambda_1=\lambda_1(\delta,\vert S\vert)<1,$ such that
$$
\parallel \varphi_k(x,y)-\varphi_k(x',y)\parallel_1\,\leq\,C_{1}\cdot\lambda_1^{(x\vert x')_y-10k\delta}
\eqno(2.22)
$$
\end{itemize}
\end{prop}
\begin{proof}
We put $\varphi_1(x,y)=f(y,x)$ in the notations of Mineyev \cite{M}, Proposition 3, pp.812-818, which we adopt from now on. We define the map $\varphi_k$ for $k>1$ as follows: $\varphi_k(x,y)=x$ if $d(x,y)\leq 10k\delta$ and 
$\varphi_k(x,y)=\varphi_k(pr_y(x),y)$ if $d(x,y)>10k\delta$ is not an integer multiple of $10\delta$. If finally 
$d(x,y)>10k\delta$ is an integer multple of $10\delta$ put
$$
\varphi_k(x,y)\,=\,\frac{1}{\sharp Fl(y,x)}\,\underset{z\in Fl(y,x)}{\sum}\,\varphi_k(pr_y(z),y).
$$
The proof of Mineyev's proposition applies to the maps $\varphi_k,k>1,$ as well and shows that the assertion holds with the same constants as in Mineyev's paper.
\end{proof}
\\

We will use Mineyev's result in our construction of a contracting homotopy of te Rips complex as follows. Instead of working with the badly behaved family of all geodesic segments joining two vertices $x$ and $y$ of the Cayley graph we will consider the family of all sequences 
$(x=x_0,x_1,\ldots,x_m=y)$ such that $x_{m-k}\in Supp(\varphi_k(x,y))$ for $0\leq k\leq\frac{d(x,y)}{10\delta}$. Each such sequence has a weight (or multiplicity), derived from the coefficients coming up in 2.11. This provides a probability measure on the space of all these "regular" sequences which allows to take averages in a sensible way.

\begin{definition}
Let $(\Gamma,S)$ be a $\delta$-hyperbolic group and let $x,y\in\Gamma$.
\begin{itemize}
\item Denote by $\Omega_{x,y}$ be the set of finite sequences 
of pairwise different vertices in $\mathcal{G}_0(\Gamma,S)=\Gamma$ beginning with $x$ and ending with $y$. The {\bf weight} of a sequence 
$$
\omega\,=\,(x_0,\,x_1,\ldots,x_m)\,\in\Omega_{x,y}
\eqno(2.23)
$$
is 
$$
c_\omega\,=\,\underset{k=1}{\overset{m}{\prod}}\,c_{k}^{x,y}(x_{m-k})
\eqno(2.24)
$$
where the coefficients on the right hand side are those of (2.19).

\item A sequence $\omega\in\Omega_{x,y}$ is called {\bf regular} if its weight is strictly positive.
\end{itemize}
\end{definition}

\begin{remark}
.\\
\begin{itemize}
\item The weights of all sequences sum up to one :  
$$
\underset{\omega\in\Omega_{x,y}}{\sum}\,c_{\omega}\,=\,1
\eqno(2.25)
$$
\item If $\omega=(x_0,\ldots,x_m)$ is a regular sequence, then 
$d(x_k,y)\,=\,10(m-k)\delta$ for $k>0$ and $d(x_0,x_1)=d(x,x_1)\leq 10\delta.$ In particular, its length equals
$$
l(\omega)\,=\,m\,\leq\,d(x,y).
\eqno(2.26)
$$
\item Let $\overline{xy}$ be the distinguished geodesic segment joining $x$ and $y$. Then
$$
Supp(\omega)\subset B(\overline{xy},\delta)
\eqno(2.27)
$$
for every regular sequence $\omega\in\Omega_{x,y}$.
\end{itemize}
\end{remark}

This is clear from 2.11 and (2.24).

\subsubsection{Orthogonal projections onto regular sequences}

\begin{definition}
Let $\omega\,=\,(x_0,\,x_1,\ldots,x_m)$ be a sequence of pairwise different vertices of $\mathcal{G}(\Gamma,S)=\Gamma$.
\begin{itemize}
\item Let $p_\omega:\,\Gamma\to\{0,\ldots,m\}$ be the map which sends $z\in\Gamma$ to the smallest index $j\in\{0,\ldots,m\}$ of an element of $Supp(\omega)$ situated at minimal distance from $z$:
$d(z,x_i)>d(z,x_j),\,0\leq i<j,\,d(z,x_j)\leq d(z,x_k),\,0\leq k\leq m$
\item Put
$\iota_\omega:\,\{0,\ldots,m\}\to\Gamma,\,i\mapsto x_i$
and let $\pi_\omega=\iota_\omega\circ p_\omega$.
\end{itemize}
\end{definition}

For a hyperbolic space the "orthogonal projection" has the following properties.

\begin{lemma}
Let $x,y,z,z'\in\Gamma$ and let $\overline{xy}$ be the distinguished geodesic segment joining $x$ and $y$. Let $\omega\in\Omega_{x,y}$ be a regular sequence. Then
$$
d(\pi_\omega(z),\pi_\omega(z'))\leq d(z,z')+24\delta.
\eqno(2.28)
$$
\end{lemma}

\begin{proof}
Let $\omega=(x_0,\ldots,x_n)\in\Omega_{x,y}$ and $z,z'\in\Gamma$ and note $u=\pi_\omega(z),u'=\pi_\omega(z')$. We may suppose that $d(z,u)\geq d(z',u')$. Let $w\in geod\{u,u'\}$. Hyperbolicity implies that $d(w,\overline{xy})\leq 3\delta$ as $Supp(\omega)\subset B(\overline{xy},\delta)$ by (2.27). Pick $w'\in\overline{xy}$ such that $d(w,w')\leq 3\delta$. If $k$ is the smallest integer such that $\vert d(w',y)-10k\delta\vert\leq  5\delta$ then 
$d(w,x_{m-k})\leq d(w,w')+d(w',\overline{yx}(10k\delta))+d(\overline{yx}(10k\delta),x_{m-k})
\leq 3\delta+5\delta+2\delta=10\delta$ and 
$d(z,w)\geq d(z,x_{m-k})-d(w,x_{m-k})\geq d(z,u)-10\delta$.
Altogether $d(z,geod\{u,u'\})\geq d(z,u)-10\delta$ and (2.5) implies then $d(z,u)\leq (u\vert u')_z+12\delta$. Thus
$$
d(z,u)\leq (u\vert u')_z+12\delta=\frac12(d(z,u)+d(z,u')-d(u,u'))+12\delta
$$
$$
\leq \frac12(d(z,u)+d(z,z')+d(z',u')-d(u,u'))+12\delta
$$
$$
\leq d(z,u)+\frac12(d(z,z')-d(u,u'))+12\delta
$$
or
$$
d(u,u')=d(\pi_\omega(z),\pi_\omega(z'))\leq d(z,z')+24\delta.
$$
\end{proof}

\subsubsection{Filling cycles near regular sequences}

We attach to any finite sequence of vertices an auxiliary contracting chain homotopy of the augmented Bar complex of $\Gamma$. This is inspired by \cite{BFS}, 4.3.

\begin{definition}
Let $(\Gamma,S)$ be a $\delta$-hyperbolic group. 
\begin{itemize}
\item For any sequence $\omega\,=\,(x_0,\,x_1,\ldots,x_m)\,\in\,\Omega_{x,y}$ let
$$
\mu^{\omega}_*\,=\,\iota_{\omega, *+1}\circ \sigma_*\circ p_{\omega, *}\,+
\,h(\pi_\omega,id)_*:\,C_*(\Gamma,{\mathbb Z})\to C_{*+1}(\Gamma,{\mathbb Z}).
\eqno(2.29)
$$
\item For $x,y\in\Gamma$ put
$$
\mu^{x,y}_*\,=\,\underset{\omega\in\Omega_{x,y}}{\sum}\,c_\omega\cdot\mu^\omega_*:\,C_*(\Gamma,{\mathbb Q})\to C_{*+1}(\Gamma,{\mathbb Q}).
\eqno(2.30)
$$
\end{itemize}
\end{definition}

\begin{lemma}.
\\
Let $(\Gamma,S)$ be a $\delta$-hyperbolic group and let $x,y\in\Gamma$. 
\begin{itemize}
\item The linear operators $\mu_*^{x,y},\,x,y\in\Gamma,$  introduced in the previous definition are contracting chain homotopies of the augmented Bar complex of $\Gamma$. 

\item 
The family of maps $\{\mu^{x,y}_*,\,x,y\in\Gamma\}$ is equivariant in the sense 
that the diagram
$$
\begin{array}{ccccc}
& & \mu^{x,y}_* & & \\
& C_*(\Gamma,{\mathbb Q}) & \to & C_{*+1}(\Gamma,{\mathbb Q}) & \\
& & & & \\
\pi(g) & \downarrow & & \downarrow & \pi(g) \\
 & & & & \\
& C_*(\Gamma,{\mathbb Q}) & \to & C_{*+1}(\Gamma,{\mathbb Q}) & \\
& & \mu^{gx,gy}_* & & \\ 
\end{array}
\eqno(2.31)
$$
commutes for all $g\in\Gamma$.

\item  
$$
Supp(\mu^{x,y}_0([z]))\,\subset\,B(geod\{x,y\},\delta)\cup\{z\},
\eqno(2.32)
$$
and 
$$
\parallel\mu^{x,y}_0([z])\parallel_1\,\leq\,d(x,z)+1
\eqno(2.33)
$$
for all $z\in\Gamma$.

\item 
$$
Supp(\mu^{x,y}_n(\alpha))\,\subset\,Supp(\alpha)\cup\left(B(Supp(\alpha),C_2)
\cap B(geod\{x,y\},\delta)\right)
\eqno(2.34)
$$
and
$$
\parallel\mu^{x,y}_n(\alpha)\parallel_1\,\leq\,C_{3}(\delta,n,diam(\alpha))
\eqno(2.35)
$$
for all $x,y\in\Gamma$, $k\geq 1$ and $\alpha\in\Delta_n(\Gamma)$, 
and some universal constants\\  $C_2=C_2(\delta,n,diam(\alpha),d(Supp(\alpha),geod\{x,y\}))$ and 
$C_3=C_3(\delta,n,diam(\alpha))$.

\end{itemize}
\end{lemma}

\begin{proof}
Lemma 2.7 and lemma 2.9 imply that $\mu_\omega$ is a contracting chain homotopy of the augmented Bar complex for any finite sequence $\omega$. 
It sends the canonical generator of the complex in degree -1 to the vertex
$[x_0]$ in degree zero. As a convex combination of contracting homotopies is still a contracting homotopy it follows that $\mu_{x,y}$ is a contracting homotopy as well.\\
The equivariance claim is clear because only metric properties were used in the definition of the linear operators in question.\\
Let $\omega=(x_0,\ldots,x_m)\in\Omega_{x,y}$ be a regular sequence and let $z\in\Gamma$. Suppose that $p_\omega(z)=k$.
Then by definition
$$
\mu_0^{\omega}([z])\,=\,(\iota_\omega\circ\sigma_0\circ p_\omega+h(\pi_\omega,id))([z])\,=\,\left(\underset{i=1}{\overset{k}{\sum}}\,[x_{i-1},x_{i}]\right)\,+\,[x_k,z].
\eqno(2.36)
$$
In particular 
$\parallel \mu_0^{\omega}([z])\parallel\,\leq\,d(x,y)+1$
 by (2.26) and thus
$\parallel \mu_0^{x,y}([z])\parallel\,\leq\,d(x,y)+1$ because $\mu^{x,y}$ is a convex combination of the operators $\mu^\omega,\omega\in\Omega_{x,y}$.
\\
\\
The assertion about the support of $\mu^{x,y}(\alpha_n)$ follow from (2.27) and (2.28).\\
\\
Let us estimate the $\ell^1$-norm of $\mu_n^{x,y}(\alpha)$. We obtain for any regular sequence $\omega\in\Omega_{x,y}$
$$
\parallel\mu^\omega(\alpha_k)\parallel_1\,\leq\,\parallel \iota_{\omega}\circ \sigma_n\circ p_{\omega}(\alpha)\parallel_1+\parallel 
h(\pi_\omega,id)(\alpha_n)\parallel_1
$$
$$
\leq diam(\pi_\omega(\alpha))\,+\,(n+1)
\,\leq\,
diam(\alpha)+24\delta +(n+1).
$$
The same bound holds then for the norm of $\mu^{x,y}(\alpha_n)$ 
by definition of $\mu^{x,y}$ as convex combination of the operators $\mu^\omega,\,\omega\in\Omega_{x,y}$. 
\end{proof}

For $r>0$ let $\pi_{y,r}:\,C_*(\Gamma,{\mathbb C})\to C_*(\Gamma,{\mathbb C})$ be the linear operator which leaves a simplex invariant if its support is contained in $B(y,r)$ and annihilates it otherwise. 

\begin{prop}
Let $(\Gamma,S)$ be a $\delta$-hyperbolic group and let $x,x'\in\Gamma$.
\begin{itemize}
\item[a)]
There exists a constant $C_4=C_4(\delta,\lambda_1)\geq 1$ such that
$$
\parallel\pi_{y,r}\circ(\mu^{x,y}_0-\mu^{x',y}_0)([y])\parallel_1\,\leq\,
C_4\cdot\lambda_1^{(x\vert x')_y-r}
\eqno(2.37)
$$
for all $y\in\Gamma$ and $0<r<(x\vert x')_y$ with $\lambda_1<1$  as in (2.22).
\item[b)] 
Let $\alpha\in\Delta_n(\Gamma),\,n\geq 1.$ Then
$$
\parallel\pi_{y,r}\circ(\mu^{x,y}_n-\mu^{x',y}_n)\circ\pi_{y,r}(\alpha)\parallel_1\,\leq\,
C_5(\delta,\lambda_1,n,diam(\alpha))\cdot\lambda_1^{(x\vert x')_y-r}
\eqno(2.38)
$$
for all $y\in\Gamma$ and $0<r<(x\vert x')_y-24\delta$ and some 
constant $C_5=C_5(\delta,\lambda_1,n,diam(\alpha))$.
\end{itemize}
\end{prop}

\begin{proof}
Let $\omega=(x_0,\ldots,x_m)\in\Omega_{x,y}$ be a regular sequence.
Then
$$
\mu_0^{\omega}([y])\,=\,\underset{i=0}{\overset{m-1}{\sum}}\,[x_i,x_{i+1}]
$$
by (2.36) and
$$
\pi_{y,r}(\mu_0^{\omega}([y]))\,=\,\underset{j=0}{\overset{k-1}{\sum}}\,[x_{m-k+j},x_{m-k+j+1}]
$$
where $k$ is the largest integer such that $10\delta k\leq r$ as 
$d(x_{m-i},y)=10i\delta$ for $i<m$. So
$$
\pi_{y,r}(\mu_0^{x,y}([y]))\,=\,
\underset{j=0}{\overset{k-1}{\sum}}\left(\underset{\Omega_{j}}{\sum}\,c_j^{x,y}(a,b)[a,b]\right)
$$
where 
$$
\Omega_{j}\,=\,\{(a,b)\in\Gamma^2,\,d(a,y)=10(j+1)\delta,d(b,y)=10j\delta\}
$$
and for $(a,b)\in\Omega_j$
$$
c_j^{x,y}(a,b)\,=\,\underset{\omega\in\Omega_{x,y}^{j,a,b}}{\sum}\,c_\omega\,
=\,\underset{\omega\in\Omega_{x,y}^{j,a,b}}{\sum}\left(\underset{i=1}{\overset{m}{\prod}}\,c_{i}^{x,y}(x_{m-i})\right)
\,=\,c_{j+1}^{x,y}(a)\cdot c_{j}^{x,y}(b)
$$
where $\Omega_{x,y}^{j,a,b}\,=\,\{(x_0,\ldots,x_m)\in\Omega_{x,y}, x_{m-j-1}=a,\,x_{m-j}=b\}$.
Similarly\\
$
\pi_{y,r}(\mu^{x',y}([y]))\,=\,
\underset{j=0}{\overset{k-1}{\sum}}\left(\underset{\Omega_{j}}{\sum}\,c_j^{x',y}(a,b)[a,b]\right)
$
with
$
c_j^{x',y}(a,b)\,=\,c_{j+1}^{x',y}(a)\cdot c_{j}^{x',y}(b).
$
Thus 
$$
\pi_{y,r}(\mu^{x,y}-\mu^{x',y})([y])\,
=\,\underset{j=0}{\overset{k-1}{\sum}}\left(\underset{\Omega_j}{\sum}\,(c_j^{x,y}(a,b)-c_j^{x',y}(a,b))[a,b]\right)
$$
and
$$
\parallel\pi_{y,r}(\mu^{x,y}-\mu^{x',y})([y])\parallel_1\,=\,
\underset{j=0}{\overset{k-1}{\sum}}\underset{\Omega_j}{\sum}\,\vert c_j^{x,y}(a,b)-c_j^{x',y}(a,b)\vert
$$
$$
=\,
\underset{j=0}{\overset{k-1}{\sum}}\underset{\Omega_j}{\sum}\,\vert c_{j+1}^{x,y}(a)\cdot c_{j}^{x,y}(b)-c_{j+1}^{x',y}(a)\cdot c_{j}^{x,y}(b)\vert
$$
$$
\leq\,
\underset{j=0}{\overset{k-1}{\sum}}\underset{\Omega_j}{\sum}\,\left(\vert c_{j+1}^{x,y}(a)-c_{j+1}^{x',y}(a)\vert\cdot c_{j}^{x,y}(b)
\,+\,c_{j+1}^{x',y}(a)\cdot \vert c_{j}^{x,y}(b)-c_{j}^{x',y}(b)\vert\right)
$$
$$
\leq\,
\underset{j=0}{\overset{k-1}{\sum}}\,C_1(\lambda_1^{(x\vert x')_y-10(j+1)\delta}+\lambda_1^{(x\vert x')_y-10j\delta})\,
\leq\, C_4(\delta,\lambda_1)\cdot\lambda_1^{(x\vert x')_y-r}.
$$
Let now $\alpha\in\Delta_n(\Gamma)$ be a simplex whose support is contained in $B(y,r)$ and let 
$k$ be the largest integer such that $10k\delta\leq r+24\delta$. Now $\pi_\omega(Supp(\alpha))\subset B(y,r+24\delta)$ by (2,28), applied with $z'=y$, and $\mu^{\omega}(\alpha)$ depends therefore only on the last $k+1$ elements of $\omega$. Therefore
$$
\mu^{x,y}(\alpha)\,=\,\underset{\omega\in\Omega_{x,y}}{\sum}\,c_\omega\,\mu^{\omega}(\alpha)\,
=\,\underset{z_1,\ldots,z_k}{\sum}\,c_{1}^{x,y}(z_1)\cdot
c_{2}^{x,y}(z_{2})\cdot\ldots\cdot c_{k}^{x,y}(z_{k})\cdot\mu^{z_k,\ldots,z_1,y}(\alpha)
$$
and 
$$
\parallel(\mu^{x,y}-\mu^{x',y})(\alpha)\parallel_1\,=\,
\parallel\underset{z_1,\ldots,z_k}{\sum}\,(c_{1}^{x,y}(z_1)\cdot\ldots\cdot
c_{k}^{x,y}(z_{k})-c_{1}^{x',y}(z_1)\cdot\ldots\cdot
c_{k}^{x',y}(z_{k}))\cdot\mu^{z_k,\ldots,z_1,y}(\alpha)\parallel_1
$$
$$
\leq \underset{z_1,\ldots,z_k}{\sum}\,\vert(c_{1}^{x,y}(z_1)\cdot\ldots\cdot
c_{k}^{x,y}(z_{k})-c_{1}^{x',y}(z_1)\cdot\ldots\cdot
c_{k}^{x',y}(z_{k}))\vert\,\cdot\parallel \mu^{z_k,\ldots,z_1,y}(\alpha)\parallel_1
$$
$$
\leq\,\left(\underset{i=1}{\overset{k}{\sum}}\,\underset{z_i}{\sum}\,
\vert c_{i}^{x,y}(z_{i})-c_{i}^{x',y}(z_{i})\vert\right)
\cdot C_3(\delta,n,diam(\alpha))
$$
$$
\leq\,C_3(\delta,n,diam(\alpha))\cdot C_4\cdot \underset{i=1}{\overset{k}{\sum}}\,\lambda_1^{(x\vert x')_y-10i\delta}
$$
$$
\leq\,C_3\cdot C_4\cdot \lambda_1^{(x\vert x')_y-r}\cdot\left(\underset{j=0}{\overset{k}{\sum}}\,\lambda_1^j\right)\,
\leq\,C_5(\delta,\lambda_1,n,diam(\alpha))\cdot\lambda_1^{(x\vert x')_y-r}.
$$
\end{proof}

\subsection{Controlled contractions of Rips complexes}

It is well known that the augmented Rips complex $C_*^R(\Gamma,{\mathbb C})$ of a $\delta$-hyperbolic group is contractible if $R>>0$ is large enough (for example $R=4\delta$ suffices). 
It defines then a resolution of the constant $\Gamma$-module $\mathbb C$ by finitely generated free $\Gamma$-modules. Therefore the Rips chain complex is an equivariant deformation retract of the Bar chain complex: there exists a $\Gamma$-equivariant chain map $\eta_*:C_*(\Gamma,{\mathbb C})\to C_*^R(\Gamma,{\mathbb C})$ whose restriction to $C_*^R(\Gamma,{\mathbb C})$ equals the identity and such that its composition with the inclusion $C_*^R(\Gamma,{\mathbb C})\hookrightarrow C_*(\Gamma,{\mathbb C})$ is equivariantly chain homotopic to the identity. Moreover, in a fixed degree, each such map is of uniformly bounded propagation and has uniformly bounded matrix coefficients because $\Delta_*^R(\Gamma)$ consists only of finitely many $\Gamma$-orbits in each degree. We fix such an equivariant deformation retraction $\eta_*$.\\
\\
Recall that the matrix coefficients of a linear operator 
$\varphi:C_m(\Gamma,{\mathbb C})\to C_n(\Gamma,{\mathbb C})$ are the unique scalars 
$\langle \varphi(\alpha),\beta\rangle,\,\alpha\in\Delta_m(\Gamma),\beta\in\Delta_n(\Gamma),$ satisfying
$$
\varphi(\alpha)\,=\,\underset{\beta\in\Delta_n(\Gamma)}{\sum}\,
\langle \varphi(\alpha),\beta\rangle\cdot\beta,\,\,\,\forall\alpha\in \Delta_m(\Gamma).
\eqno(2.39)
$$

{\vskip 3mm}

Our first main result is

\begin{theorem}
Let $(\Gamma,S)$ be a $\delta$-hyperbolic group and let $R\geq 12\delta$ be an even integer so that the augmented Rips-complex $C^R_*(\Gamma,{\mathbb C})$ is contractible. Let $\eta_*:\,C_*(\Gamma,{\mathbb C})\to C^R_*(\Gamma,{\mathbb C})$ be a $\Gamma$-equivariant deformation retraction of the Bar complex onto the Rips complex of scale $R$.
\begin{itemize}
\item There exists a family of linear operators
$$
h^{x}_*:\,C_*^R(\Gamma,{\mathbb Q})\,\to\,C_{*+1}^R(\Gamma,{\mathbb Q}),\,\,\,x\in\Gamma,\,\,*\geq -1,
\eqno(2.40)
$$
on the augmented Rips-complex satisfying the identity
$$
h^{x}_*\circ\partial+\partial\circ h^{x}_*\,=\,id,
\eqno(2.41)
$$
and such that the following assertions hold.  
\item The operators $(h_*^x)_{x\in\Gamma}$ are
compatible with the group action in the sense that the diagrams
$$
\begin{array}{ccccc}
 & & h^{x}_* & & \\
 & C_*^R(\Gamma,{\mathbb Q}) & \to & C_{*+1}^R(\Gamma,{\mathbb Q}) & \\
 & & & & \\
 \pi(g) & \downarrow & & \downarrow & \pi(g) \\
 & & & & \\
 & C_*^R(\Gamma,{\mathbb Q}) & \to & C_{*+1}^R(\Gamma,{\mathbb Q}) & \\
 & & h^{gx}_* & & \\
\end{array}
\eqno(2.42)
$$
commute for all $g\in\Gamma$ and all $x\in X$. 
\item The matrix coefficients 
$\langle h_n^{x}(\alpha),\beta\rangle$ vanish unless 
$$
Supp(\beta)\subset B(geod((Supp\,\alpha)\cup\{x\}),C_6(\delta,R,\eta,n))
\eqno(2.43) 
$$
and satisfy the estimates
$$
\vert\langle h_n^{x}(\alpha),\beta\rangle\vert\,\leq\,C_7(\delta,\vert S\vert,R,\eta,n)
\eqno(2.44)
$$
for suitable constants $C_6=C_6(\delta,R,\eta,n)$ and $C_{7}=C_{7}(\delta,\vert S\vert,R,\eta,n)$.

\item The homotopy operators depend only weakly on the choice of the base point in the sense that
$$
\vert\langle (h_n^{x}-h_n^{x'})(\alpha),\beta\rangle\vert\,\leq\,C_{8}(\delta,\vert S\vert,R,\eta,\lambda_1,n,d(x,x'))\cdot \lambda_1^{(x\vert x')_\beta}
\eqno(2.45)
$$
with $(x\vert x')_\beta\,=\,\underset{z\in Supp(\beta)}{Min}\,(x\vert x')_z$ for all $x,x'\in\Gamma,\,\alpha\in\Delta_n^R(\Gamma),\,\beta\in\Delta_{n+1}^R(\Gamma)$
and for a suitable constant 
$$
C_{8}=C_{8}(\delta,\vert S\vert,R,\eta,\lambda_1,n,d(x,x')).
$$ 
\end{itemize}
\end{theorem}

\begin{proof}
We construct the operator $h^x$ by induction over the degree.
We put $h^x_{-1}(1)=[x]$ in degree -1. Suppose now that $h_*^x$ has been defined up to degree $*=n-1$ and let $\alpha=[x_0,\ldots,x_n]\in\Delta_n^R(\Gamma)$. Then $(id-h_{n-1}^x\circ\partial)(\alpha)$ is a cycle in $C_{n-1}^R(\Gamma)$. We put
$$
h_n^x(\alpha)\,=\,\eta_{n+1}\circ\mu_n^{x,x_0}\circ(id-h^x_{n-1}\circ\partial)(\alpha).
\eqno(2.46)
$$
Then 
$$
\partial\circ h_n^x(\alpha)\,=\,\eta_{n+1}\circ(id-h^x_{n-1}\circ\partial)(\alpha)
\,=\,(id-h^x_{n-1}\circ\partial)(\alpha)
$$
by the induction assumption because $\mu^{x,x_0}$ is a contracting chain homotopy of the Bar-complex and $\eta_*$ is a chain map which equals the identity on the Rips complex $C_*^R(\Gamma,{\mathbb C})$.
This shows our first claim. The second follows from (2.31) and the equivariance of $\eta_*$. 
Concerning our third claim note first that 
$$
h^x_0([x_0])\,=\,\eta\circ\mu_0^{x,x_0}([x_0]-[x])\,=\,
\mu_0^{x,x_0}([x_0])
$$ 
because $\mu_0^{x,x_0}([x_0])$ is a linear combination of edges of length at most $12\delta\leq R$. Thus $Supp(h^x_0([x_0]))\,\subset\,B(geod\{x,x_0\},\delta)$ by (2.32). Assertion (2.43) follows now by induction from (2.34) and the fact that $\eta_*$ is of uniformly bounded propagation in each degree.\\ The matrix coefficients of $h^x_0$ are bounded by 1 according to (2.36). 
Suppose that claim (2.44) has been verified up to degree $n-1$ and let $\alpha=[x_0,\ldots,x_n]\in\Delta_n^R(\Gamma),\,n>0$. The estimate (2.35) shows that the restriction of the operator $\mu_n^{x,x_0}$ to any Rips-subcomplex of the Bar-complex is bounded w.r.t $\ell^1$-norms. The same holds for the operator $\eta$ because it is of uniformly bounded propagation and has uniformly bounded matrix coefficients in a fixed degree. Thus
$$
\vert\langle h^x_n(\alpha),\beta\rangle\vert\,\leq\,\underset{\beta'}{\sum}\,
\vert\langle \eta_{n+1}\circ\mu_n^{x,x_0}(\beta'),\beta\rangle\vert\cdot\vert\langle
(id-h^x_{n-1}\circ\partial)(\alpha),\beta'\rangle\vert
$$
$$
\leq\,\parallel\eta_{n+1}\parallel_1\cdot\, C_3(\delta,n,R)\cdot (1+(n+1)\cdot C_7(\delta,\vert S\vert,R,\eta,n-1))
$$
$$
\cdot\vert\{\beta'\in\Delta_n^R(\Gamma),\,\langle \eta_{n+1}\circ\mu_n^{x,x_0}(\beta'),\beta\rangle\cdot\langle (id-h^x_{n-1}\circ\partial)(\alpha),\beta'\rangle\neq 0\}\vert
\eqno(2.47)
$$
by our induction hypothesis. For every simplex $\beta'\in\Delta^R_{n}(\Gamma)$ in this set the distance 
$d(Supp(\beta'),geod\{x,Supp(\alpha)\})$ is bounded in terms of $\delta,R,\eta,n$ by (2.43). Assertion (2.34) and the fact that $\eta_{n+1}$ is of uniformly bounded propagation imply then that the distance 
$d(Supp(\beta),Supp(\beta'))$ is bounded in terms of $\delta,R,\eta,n$. This implies that the cardinality of the set (2.47) is bounded in terms of $\delta,\vert S\vert,R,\eta,n$ and finishes the proof of (2.44).\\ We come now to the key estimate (2.45) and establish it first in degree zero. 
So let $\alpha=[y]$ and $\beta=[x_0,x_1]$ be such that $\langle (h_n^{x}-h_n^{x'})(\alpha),\beta\rangle\neq 0$. Then $\{x_0,x_1\}\subset B(geod\{x,y\},\delta)\cup B(geod\{x',y\},\delta)$ by (2.32). Suppose that $(x\vert x')_\beta>4\delta$. Then $\{x_0,x_1\}\subset B(geod\{x,y\},2\delta)\cap B(geod\{x',y\},2\delta)$ by hyperbolicity and a simple calculation shows that 
$$
(x\vert x')_\beta\,\leq\,(x\vert x')_y-max(d(x_0,y),d(x_1,y))+4\delta.
$$ 
We may apply (2.37) with $r=max(d(x_0,y),d(x_1,y))$ and find 
$$
\vert\langle (h_n^{x}-h_n^{x'})(\alpha),\beta\rangle\vert\,\leq\,C_4\cdot\lambda_1^{(x\vert x')_y-max(d(x_0,y),d(x_1,y))}\,\leq\,C_4(\delta,\lambda_1)\cdot\lambda_1^{(x\vert x')_\beta-4\delta}.
$$ 
If $(x\vert x')_\beta\leq 4\delta$ we may choose $C_8$ sufficiently large to ensure (2.45). 
Suppose now that our claim has been verified up to degree $n-1$ and let $\alpha=[x_0,\ldots,x_n]\in\Delta_n^R(\Gamma)$. 
Then 
$$
(h_n^{x}-h_n^{x'})(\alpha)\,=\,\eta_{n+1}\circ\mu^{x,x_0}\circ(id-h^x_{n-1}\circ\partial)(\alpha)-
\eta_{n+1}\circ\mu^{x',x_0}\circ(id-h^{x'}_{n-1}\circ\partial)(\alpha)
$$
$$
=\,\eta_{n+1}\circ(\mu^{x,x_0}-\mu^{x',x_0})\circ(id-h^x_{n-1}\circ\partial)(\alpha)\,+\,
\eta_{n+1}\circ\mu^{x',x_0}\circ(h^{x'}_{n-1}-h^x_{n-1})\circ\partial(\alpha)
$$
The induction assumption and the same reasoning as in the proof of (2.44) show that the second term in the previous sum satisfies (2.45). So it remains to bound the first term. One finds
$$
\vert\langle\eta_{n+1}\circ(\mu^{x,x_0}-\mu^{x',x_0})\circ(id-h^x_{n-1}\circ\partial)(\alpha),\beta\rangle\vert
$$
$$
\leq\,\underset{\beta'}{\sum}\,
\vert\langle \eta_{n+1}\circ(\mu^{x,x_0}-\mu^{x',x_0})(\beta'),\beta\rangle\vert\cdot\vert\langle
(id-h^x_{n-1}\circ\partial)(\alpha),\beta'\rangle\vert.
$$
Let $\beta'\in\Delta^R_n(\Gamma)$ be a simplex such that
$$
\langle \eta_{n+1}\circ(\mu^{x,x_0}-\mu^{x',x_0})(\beta'),\beta\rangle\cdot\langle
(id-h^x_{n-1}\circ\partial)(\alpha),\beta'\rangle\neq 0.
$$
Then the support of $\beta'$ is contained in a tubular neighbourhood of $geod(Supp(\alpha)\cup\{x\})$ whose width is controlled in terms of $(\delta,R,\eta,n-1)$ by (2.43). Consequently
$$
Supp(\beta)\,\subset\,Supp(\eta_{n+1}\circ(\mu^{x,x_0}-\mu^{x',x_0})(\beta'))\,
\,\subset\,B(Supp(\beta'),C_9(\delta,n,\eta,R))
$$
by (2.34) and the finite propagation of $\eta$ in each degree provided that $(x\vert x')>\delta$. The same calculation as before yields the estimate
$$
(x\vert x')_\beta\,\leq\,(x\vert x')_{x_0}-\underset{z\in Supp(\beta)}{max}\,d(x_0,z)+C_{10}(\delta,n,\eta,R,d(x,x'))
$$
for a suitable constant $C_{10}=C_{10}(\delta,n,\eta,R,d(x,x'))$. Suppose now that\\ $(x\vert x')_\beta\,>\,C_9+R+C_{10}+24\delta$ and put 
$$
r_0=\underset{z\in Supp(\beta)}{max}\,d(x_0,z)+C_{9}(\delta,n,\eta,R,d(x,x'))+R.
$$ 
Then 
$Supp(\beta)\cup Supp(\beta')\subset B(x_0,r_0)$ and
$$
r_0\leq (x\vert x')_{x_0}-(x\vert x')_\beta+C_9+R+C_{10}<(x\vert x')_{x_0}-24\delta
$$
so that we may apply (2.38) to conclude that
$$
\vert\eta_{n+1}\circ(\mu^{x,x_0}-\mu^{x',x_0})(\beta'),\beta\rangle\vert\,\leq\,C_{11}(\delta,\lambda_1,R,\eta,n,d(x,x'))\cdot\lambda_1^{(x\vert x')_{x_0}-r_0}
$$
$$
\leq\,C_{11}(\delta,\lambda_1,R,\eta,n,d(x,x'))\cdot\lambda_1^{(x\vert x')_{\beta}-C_9-R-C_{10}}
\,\leq\,C_{12}(\delta,\lambda_1,R,\eta,n,d(x,x'))\cdot\lambda_1^{(x\vert x')_{\beta}}.
$$
If $(x\vert x')_\beta\,\leq\,C_9+R+C_{10}+24\delta$ the same bound can be derived, after possibly enlarging $C_{12}$, from (2.35). Assertion (2.44) and the same counting argument as in its proof allow to conclude (2.45).
\end{proof}

\section{The bimodule of Lafforgue}

\subsection{Continuous metrics and Lafforgue's operator}
Let $(\Gamma,S)$ be a $\delta$-hyperbolic group. 
For a metric $\widehat{d}$ on $\Gamma$ which is quasiisometric to the word metric $d=d_S$ put
$$
\varrho_{\widehat{d}}:\,{\mathbb R}_+\to{\mathbb R}_+,
\,\,\,
\varrho\,_{\widehat{d}}(r)\,=\,\Sup\,\vert\widehat{d}(x,y)-\widehat{d}(x',y)-\widehat{d}(x,y')+\widehat{d}(x',y')\vert,
\eqno(3.1)
$$
where the supremum is taken over all $x,x',y,y'\in\Gamma$ satisfying
$d(x,x')=d(y,y')=1$ and $d(\{x,x'\},\{y,y'\})\geq r$.
\\
As observed by various authors, it is possible to replace 
the word metric on a hyperbolic group by a "continuous" metric, quasiisometric to the original one, for which the difference of the distances of two adjacent vertices from a base point far away becomes almost independent of the choice of the base point. In other words, for such a metric the function $\varrho\,_{\widehat{d}}$ vanishes at infinity. It is this continuity property, which, apart from the particular choice of the contracting homotopy, assures that the Lafforgue-triple is actually a Kasparov-bimodule.\\
\\
The "continuous metrics" we have to deal with are the following.

\begin{theorem}(Mineyev-Yu)\cite{MY}
Let $(\Gamma,S)$ be a $\delta$-hyperbolic group. There exists 
a $\Gamma$-equivariant distance $d_{MY}:\Gamma\times\Gamma\to{\mathbb Q}_+$, which is quasi-isometric to the word metric and such that
$$
\varrho_{d_{MY}}(r)\,\leq\,C_{13}(\delta,\vert S\vert)\cdot
\lambda_2(\delta,\vert S\vert)^r
\eqno(3.2)
$$
for suitable universal constants $\lambda_2=\lambda_2(\delta,\vert S\vert)<1$ and $C_{13}=C_{13}(\delta,\vert S\vert)$.
Moreover, there exists a universal constant $C_{14}=C_{14}(\delta,\vert S\vert)>0$ such that
$$
d_{MY}(x,z)+d_{MY}(z,y)\,\leq\,d_{MY}(x,y)+C_{14}
\eqno(3.3)
$$
whenever $z\in geod(x,y),\,x,y,z\in\Gamma$.
\end{theorem}

\begin{proof}
In view of the definition of the Mineyev-Yu metric $\widehat{d}$ in terms of the functions
$s$ and $r$ introduced in \cite{MY}, pp. 115-116, the $\Gamma$-equivariance of $\widehat{d}$ and  assertions (3.2) follows from \cite{MY}, Theorem 6. Assertion (3.3) is an immediate consequence of \cite{MY}, Proposition 10 b).\\
\end{proof}

\begin{theorem}(Lafforgue)\cite{La2}, Section 3.5
Let $(\Gamma,S)$ be a $\delta$-hyperbolic group. There exists 
a $\Gamma$-equivariant distance $d_{Laff}:\Gamma\times\Gamma\to{\mathbb Q}_+$ quasi-isometric to the word-metric such that
$$
\varrho_{d_{Laff}}(r)\,\leq\,\frac{C_{15}(\delta,\vert S\vert)}{1+r}.
\eqno(3.4)
$$
for a suitable universal constant $C_{15}=C_{15}(\delta,\vert S\vert)$.
\end{theorem}
Lafforgue's metric also satisfies an estimate similar to (3.3).\\
So the characteristic function $\varrho_{\widehat{d}}$ decays polynomially for the Lafforgue-metric, but exponentially for the metric of Mineyev-Yu. It is this exponential decay, which, together with the results of section 1, guarantees the finite summability of the modified Lafforgue-bimodules.

\begin{definition} (Lafforgue) \cite{La2}, p.69, 4.4.
Let $(\Gamma,S)$ be a $\delta$-hyperbolic group.\\
Let $\widehat{d}:\Gamma\times\Gamma\to{\mathbb R}_+$ be a metric on $\Gamma$ and let $x\in\Gamma$ be a base point. 
\begin{itemize}
\item[a)] For $t\in{\mathbb R}$ put
$$
\begin{array}{cccc}
e^{t\widehat{d}_x}: & C_*(\Gamma,{\mathbb R}) & \to & C_*(\Gamma,{\mathbb R}) \\
& & & \\
& \alpha=[x_0,\ldots,x_n] & \mapsto & e^{t\cdot\widehat{d}(x,x_0)}\cdot\alpha.
\end{array}
\eqno(3.5)
$$
\item[b)]
Let $h^x:C_*^R(\Gamma,{\mathbb C})\to C_{*+1}^R(\Gamma,{\mathbb C})$ 
be a contracting chain homotopy of the augmented Rips complex as constructed in 2.19.
For $t>0$ put
$$
\begin{array}{cccc}
\Phi^{x,t}_*\,=\,e^{t\widehat{d}_x}\circ h^x_*\circ e^{-t\widehat{d}_x}: & C_*(\Gamma,{\mathbb R}) & \to & C_{*+1}(\Gamma,{\mathbb R}) \\
\end{array}
\eqno(3.6)
$$
\end{itemize}
\end{definition}

\subsection{Estimates of matrix coefficients}

\begin{prop}
Let $(\Gamma,S)$ be a $\delta$-hyperbolic group and let $x,x'\in\Gamma$.
Let $h^x$, $x\in\Gamma,$ 
be a contracting chain homotopy of the Rips chain complex as constructed in Theorem 2.19. and let $\widehat{d}:\Gamma\times\Gamma\to{\mathbb Q}$ be a $\Gamma$-equivariant distance on $\Gamma$ quasi-isometric to the word metric and satisfying (3.3). Then 

$$
\vert\langle \Phi^{x,t}(\alpha),\beta\rangle\vert\,
\leq\,
C_{16}(\delta,\vert S\vert, R,t,\eta,\widehat{d},n)\cdot e^{-\lambda_3\cdot d(x_0,y_0)\cdot t}
\eqno(3.7)
$$

and

$$
\vert\langle(\Phi^{x,t}-\Phi^{x',t})(\alpha),\beta\rangle\vert\,\leq\,
$$
$$
\leq C_{17}\cdot d(x,x')\cdot d(x_0,y_0)\cdot
e^{-\lambda_3\cdot d(x_0,y_0)\cdot t}
\cdot\left(\varrho_{\widehat{d}}\,((x\vert x')_\beta-C_{18})+\lambda_1^{(x\vert x')_\beta}\right)
\eqno(3.8)
$$

{\vskip 5mm}

for all $x,x'\in\Gamma,\,t>0,\,\alpha=[x_0,\ldots,x_n]\in\Delta_n^R(\Gamma),\beta=[y_0,\ldots,y_{n+1}]\in\Delta_{n+1}^R(\Gamma)$ and suitable constants
$C_{16}=C_{16}(\delta,\vert S\vert, R,t,\eta,\widehat{d},n),$
$\lambda_3=\lambda_3(\widehat{d})<1$,\\
$C_{17}=C_{17}(\delta,\vert S\vert, R,t,\eta,\widehat{d},n,d(x,x'))$ and $
C_{18}=C_{18}(\delta,R,\eta,n).$
\end{prop}

\begin{proof}
Let $\alpha=[x_0,\ldots,x_n]\in\Delta_n^R(\Gamma)$ and $\beta=[y_0,\ldots,y_{n+1}]\in\Delta_{n+1}^R(\Gamma)$ be Rips-simplices. Then
$$
\langle \Phi^{x,t}_n(\alpha),\beta\rangle\,=\,
\langle e^{t\widehat{d}_x}\circ h^x_*\circ e^{-t\widehat{d}_x}(\alpha),\beta\rangle\,=
\,e^{t(\widehat{d}(x,y_0)-\widehat{d}(x,x_0))}
\cdot \langle h^x_n(\alpha),\beta\rangle
$$
According to (2.43) the distance $d(y_0,geod\{x,x_0\})$ is bounded in terms of $\delta,R,\eta,n$. So we may  find $z\in geod\{x,x_0\}$ such that $d(y_0,z)\leq C(\delta,R,\eta,n)$. Then
$$
\widehat{d}(x,y_0)-\widehat{d}(x,x_0)\,\leq\,\widehat{d}(x,z)-\widehat{d}(x,x_0)+\widehat{d}(z,y_0)
$$
$$
\leq\,-\widehat{d}(x_0,z)+C_{14}(\widehat{d})+C(\delta,R,\eta,n)\,\leq\,
-\lambda_3(\widehat{d})\cdot d(x_0,y_0)+C_{19}(\delta,R,\eta,n,\widehat{d})
\eqno(3.9)
$$
because $\widehat{d}\leq\lambda_3\cdot d+C'$ for suitable $\lambda_3=\lambda_3(\widehat{d}),\,C'=C'(\widehat{d})$ and (3.3) holds. 
The assertion follows then from (2.44).\\
\\
We consider now the operator 
$$
\Phi^{x,t}-\Phi^{x',t}\,=\,e^{t\widehat{d}_x}\circ (h^x_*-h^{x'}_*)\circ e^{-t\widehat{d}_x}\,+\,
(e^{t\widehat{d}_x}\circ h^{x'}_*\circ e^{-t\widehat{d}_x}
-e^{t\widehat{d}_{x'}}\circ h^{x'}_*\circ e^{-t\widehat{d}_{x'}})
$$
and estimate the matrix coefficients of the two terms separately. We may suppose without loss of generality that $(x\vert x')_\beta>\delta$. On the one hand we derive from 
(2.45) by the previous argument
$$
\vert\langle e^{t\widehat{d}_x}\circ (h^x_n-h^{x'}_n)\circ e^{-t\widehat{d}_x}(\alpha),\beta\rangle\vert\,\leq\,C_{20}(\delta,\vert S\vert,\lambda_1,\eta,n,\widehat{d},d(x,x'),t)e^{-\lambda_3\cdot d(x_0,y_0)\cdot t}\cdot\lambda_1^{(x\vert x')_\beta}
$$
On the other hand 
$$
\vert\langle(e^{t\widehat{d}_x}\circ h^{x'}_n\circ e^{-t\widehat{d}_x}
-e^{t\widehat{d}_{x'}}\circ h^{x'}_n\circ e^{-t\widehat{d}_{x'}})(\alpha),\beta\rangle\vert
$$
$$
=\,\vert e^{t(\widehat{d}(x,y_0)-\widehat{d}(x,x_0))}-e^{t(\widehat{d}(x',y_0)-\widehat{d}(x',x_0))}\vert\cdot\vert\langle h^{x'}_n(\alpha),\beta\rangle\vert.
$$
The inequality
$$
\vert e^b-e^a\vert\,=\,\,\vert\underset{a}{\overset{b}{\int}}\,e^s ds\vert\,\leq\,
e^{max(a,b)}\vert b-a\vert,
$$
valid for any $a,b\in{\mathbb R}$ 
and the bound
$$
Max\left( (\widehat{d}(x,y_0)-\widehat{d}(x,x_0)),(\widehat{d}(x',y_0)-\widehat{d}(x',x_0))\right)
\,\leq\,-\lambda_3\cdot d(x_0,y_0)\,+\,C_{20}(\delta,R,\eta,\widehat{d},n),
$$
which follows from (3.9) because $(x\vert x')_\beta>\delta$ 
lead then to the estimate
$$
\vert e^{t(\widehat{d}(x,y_0)-\widehat{d}(x,x_0))}-e^{t(\widehat{d}(x',y_0)-\widehat{d}(x',x_0))}\vert\cdot\vert\langle h^{x'}_n(\alpha),\beta\rangle\vert
$$
$$
\leq\,C_{21}(\delta,R,\eta,\widehat{d},n,t)\cdot\vert(\widehat{d}(x,y_0)-\widehat{d}(x,x_0))-(\widehat{d}(x',y_0)-\widehat{d}(x',x_0))\vert\cdot e^{-\lambda_3\cdot d(x_0,y_0)\cdot t}.
$$
Fix geodesic segments $\overline{xx'}$ and $\overline{x_0y_0}$ with consecutive vertices\\ $x=u_0,\ldots,u_i,\ldots u_k=x'$ and $y_0=v_0,\ldots,v_j,\ldots,v_l=x_0$.
The estimates 
$$
\vert (\widehat{d}(u_i,v_j)-\widehat{d}(u_i,v_{j+1}))-(\widehat{d}(u_{i+1},v_j)-\widehat{d}(u_{i+1},v_{j+1}))\vert\,\leq\,\varrho_{\widehat{d}}(d(u_i,v_j)),
$$
$0\leq i<k,\,0\leq j<l,$ which hold by definition allow to obtain the bound
$$
\vert(\widehat{d}(x,y_0)-\widehat{d}(x,x_0))-(\widehat{d}(x',y_0)-\widehat{d}(x',x_0))\vert
\,\leq\,d(x,x')\cdot d(y_0,x_0)\cdot\underset{z\in geod\{y_0,x_0\}}{Max}\,\varrho_{\widehat{d}}((x\vert x')_z)
$$

$$
\leq\,d(x,x')\cdot d(y_0,x_0)\cdot\varrho_{\widehat{d}}((x\vert x')_{y_0}-C_{18}(\delta,R,\eta,n))
$$
where we used the fact that the characteristic function $\varrho_{\widehat{d}}$ is monotone decreasing. Our claim follows now from (2.44).
\end{proof}

Similarly one obtains 

\begin{lemma}
$$
\vert\langle(e^{t\widehat{d}_x}\circ\partial\circ e^{-t\widehat{d}_x}-e^{t\widehat{d}_{x'}}\circ\partial\circ e^{-t\widehat{d}_{x'}})(\alpha),\beta\rangle\vert\,\leq\,C_{22}\cdot d(x,x')\cdot
\varrho_{\widehat{d}}\,((x\vert x')_\beta-R)
\eqno(3.10)
$$
and
$$
\vert\langle(e^{t\widehat{d}_x}\circ\pi_{alt}\circ e^{-t\widehat{d}_x}-e^{t\widehat{d}_{x'}}\circ\pi_{alt}\circ e^{-t\widehat{d}_{x'}})(\alpha),\beta'\rangle\vert\,\leq\,C_{22}\cdot d(x,x')\cdot
\varrho_{\widehat{d}}\,((x\vert x')_\beta-R)
\eqno(3.11)
$$
for all $x,x'\in\Gamma,\,t>0,\,\alpha=[x_0,\ldots,x_n],\beta'=[y_0,\ldots,y_n]\in\Delta_n^R(\Gamma),$\\ $\beta=[y_0,\ldots,y_{n-1}]\in\Delta_{n-1}^R(\Gamma)$ and a suitable constant $C_{22}=C_{22}(R,\widehat{d},t)$.
\end{lemma}

\begin{lemma}
Let $x,x'\in\Gamma$ and $n\geq 0$. If $t>>0$ is sufficiently large the linear maps (3.6) extend to bounded linear operators 
$$
\Phi^{x,t}_n:\,\ell^2(\Delta_n^R(\Gamma))\to \ell^2(\Delta_{n+1}^R(\Gamma)).
$$ 
If moreover $\underset{r\to\infty}{\lim}\,\varrho(\widehat{d})(r)\,=\,0$, then 
$$
\Phi^{x,t}_n-\Phi^{x',t}_n\,\in\,{\mathcal K}(\ell^2(\Delta_n^R(\Gamma)),\ell^2(\Delta_{n+1}^R(\Gamma))).
$$
\end{lemma}

\begin{proof}
Every linear map $T:C_n^R(\Gamma,{\mathbb C})\to C_m^R(\Gamma,{\mathbb C})$ is determined by its matrix coefficients $T(\alpha)=\underset{\beta}{\sum}\,c_{\alpha\beta}\cdot\beta,\,\alpha\in\Delta_n^R(\Gamma),\,\beta\in\Delta_m^R(\Gamma)$. For an integer $r\geq 0$ we define a linear map 
$T(r):C_n^R(\Gamma,{\mathbb C})\to C_m^R(\Gamma,{\mathbb C}),\,T(r)(\alpha)=\underset{\beta}{\sum}\,c^{(r)}_{\alpha\beta}\cdot\beta,$ by the condition 
$$
c^{(r)}_{\alpha\beta}\,=\,\begin{cases}
c_{\alpha\beta}, & d'(\alpha,\beta)=r, \\
0 & d'(\alpha,\beta)\neq r \\
\end{cases}
$$
for $\alpha=[x_0,\ldots,x_n],\,\beta=[y_0,\ldots,y_m]$, where $d'(\alpha,\beta)=d(x_0,y_0)$. So $T(r)$ is the component of propagation $r$ of $T$ and $T\,=\underset{r}{\sum}\,T(r)$ pointwise. We are interested in the case $T=\Phi^{x,t}_n$ and want to estimate the operator norm of $T(r)$.

One finds for $\xi=\underset{\alpha\in \Delta_n^R(\Gamma)}{\sum}\,\xi_\alpha\cdot \alpha$ and $\eta=\underset{\beta\in\Delta_{(n+1)}^R(\Gamma)}{\sum}\,\eta_\beta\cdot \beta$
$$
\vert\langle T(r)\xi,\eta\rangle\vert\,=\,\vert\underset{\alpha,\beta}{\sum}\,
\xi_\alpha\cdot c^{(r)}_{\alpha\beta}\cdot\overline{\eta_\beta}\vert\,\leq\,
\underset{\alpha,\beta}{\sum}\,
\vert\xi_\alpha\vert\cdot\vert c^{(r)}_{\alpha\beta}\vert\cdot\vert\eta_\beta\vert
$$

$$
=\,\underset{\alpha,\beta}{\sum}\,
(\vert\xi_\alpha\vert\cdot\vert c^{(r)}_{\alpha\beta}\vert^{\frac12})\cdot(\vert c^{(r)}_{\alpha\beta}\vert^{\frac12}\cdot\vert\eta_\beta\vert)
$$

$$
\leq\,\left(\underset{\alpha,\beta}{\sum}\,
\vert\xi_\alpha\vert^2\cdot\vert c^{r}_{\alpha\beta}\vert\right)^{\frac12}\cdot\left(\underset{\alpha',\beta'}{\sum}\,\vert c^{(r)}_{\alpha'\beta'}\vert\cdot\vert\eta_{\beta'}\vert^2\right)^{\frac12}
$$

by the Cauchy-Schwarz inequality. Now 
$$
\vert\{\beta\in\Delta_m^R(\Gamma),\,d'(\alpha,\beta)=r\}\vert\,\leq\,\vert S\vert^r\cdot (1+\vert S\vert)^{(n+1)R}
\eqno(3.12)
$$
for all $\alpha\in\Delta_n^R(\Gamma)$
so that
$$
\left(\underset{\alpha,\beta}{\sum}\,
\vert\xi_\alpha\vert^2\cdot\vert c^{r}_{\alpha\beta}\vert\right)\,\leq\,C_{16}\cdot e^{-\lambda_3\cdot r\cdot t}\cdot 
\vert S\vert^r\cdot (1+\vert S\vert)^{mR}\cdot\parallel\xi\parallel^2
$$
by (3.7).
Similarly
$$
\vert\{\alpha\in\Delta_n^R(\Gamma),\,d'(\alpha,\beta)=r\}\vert\,\leq\,\vert S\vert^r\cdot (1+\vert S\vert)^{nR}
\eqno(3.13)
$$
for all $\beta\in\Delta_m^R(\Gamma)$
so that
$$
\left(\underset{\alpha',\beta'}{\sum}\,\vert c^{(r)}_{\alpha'\beta'}\vert\cdot\vert\eta_{\beta'}\vert^2\right)\,\leq\,C_{16}\cdot e^{-\lambda_3\cdot r\cdot t}\cdot 
\vert S\vert^r\cdot (1+\vert S\vert)^{nR}\cdot\parallel\eta\parallel^2
$$
and
$$
\vert\langle T\xi,\eta\rangle\vert\,\leq\,C_{16}\cdot e^{-\lambda_3\cdot r\cdot t}\cdot\vert S\vert^r\cdot (1+\vert S\vert)^{\frac{(2n+1)R}{2}}\cdot\parallel\xi\parallel\cdot\parallel\eta\parallel.
$$
Thus
$$
\parallel\Phi_n^{x,t}(r)\parallel\,\leq\,
C_{16}(\delta,\vert S\vert, R,t,\eta,\widehat{d},n)\cdot e^{-\lambda_3\cdot r\cdot t}\cdot\vert S\vert^r\cdot (1+\vert S\vert)^{\frac{2n+1)R}{2}}.
$$
So the linear maps $\Phi^{x,t}_n(r)$ extend to bounded linear operators on the corresponding $\ell^2$-spaces. If moreover $t>\lambda_3^{-1}\cdot\log(\vert S\vert)$ the series $\underset{r}{\sum}\,\Phi^{x,t}_n(r)$ converges in\\ ${\mathcal L}(\ell^2(\Delta_n^R(\Gamma)),\ell^2(\Delta_{n+1}^R(\Gamma)))$ to a bounded linear operator which extends $\Phi^{x,t}_n$.\\
\\
Let now $\pi_{r'}\in{\mathcal L}(\ell^2(\Delta_{n+1}^R(\Gamma)))$ be the orthogonal projection onto the linear span of the finitely many simplices $\beta\in\Delta_{n+1}^R(\Gamma)$ satisfying $(x\vert x')_\beta\leq r'$. The previous argument and (3.8) show then that for $t>\lambda_3^{-1}\cdot \log(\vert S\vert)$
$$
\underset{r'\to\infty}{\lim}\,\parallel(Id-\pi_{r'})\circ(\Phi^{x,t}_n-\Phi^{x',t}_n)\parallel\,=\,0
$$
which implies that $\Phi^{x,t}_n-\Phi^{x',t}_n$ is a compact operator in this case.
\end{proof}

\begin{prop}
Suppose that $\varrho(\widehat{d})(r)\,=\,O(\lambda^r)$ for some $\lambda<1$. 
Then 
$$
\Phi^{x,t}_n-\Phi^{x',t}_n\,\in\,\ell^p(\ell^2(\Delta_n^R(\Gamma)),\ell^2(\Delta_{n+1}^R(\Gamma)))
$$
for $t>>0$ and  $p>>0$ sufficiently large.
\end{prop}

\begin{proof}
The notations are the same as in the proof of the previous proposition. 
We want to estimate the Schatten $p$-norm of the operators $T(r)=\Phi^{x,t}_n(r)-\Phi^{x',t}_n(r)$ for $r\geq 0, t>0$ and $p=2N>>0$ sufficiently large. To this end we study the matrix coefficients of the operators $(T^*(r)T(r))^N$. We write in the sequel $c_{\alpha\beta}$ instead of $c^{(r)}_{\alpha\beta}$. One has
$$
\vert\langle(T^*(r)T(r))^N(\alpha),\alpha\rangle\vert\,=\,
\left\vert\underset{\underset{\beta_1,\ldots,\beta_N\in\Delta_m}{\alpha_2,\ldots,\alpha_{N}\in\Delta_n}}{\sum}\,
c_{\alpha\beta_1}\cdot c_{\beta_1\alpha_2}^*\cdot c_{\alpha_2\beta_2}\cdot
 c_{\beta_2\alpha_3}^*\cdot\ldots\cdot c_{\alpha_{N}\beta_N}\cdot 
 c_{\beta_N\alpha}^*\right\vert
$$
$$
=\,\left\vert\underset{\underset{\beta_1,\ldots,\beta_N\in\Delta_m}{\alpha_2,\ldots,\alpha_{N}\in\Delta_n}}{\sum}\,
c_{\alpha\beta_1}\cdot c_{\alpha_2\beta_1}\cdot c_{\alpha_2\beta_2}\cdot
 c_{\alpha_3\beta_2}\cdot\ldots\cdot c_{\alpha_{N}\beta_N}\cdot 
 c_{\alpha\beta_N}\right\vert.
 \eqno(3.14)
$$
Now (3.8) and our assumptions imply
$$
\vert c_{\alpha'\beta'}\vert\leq C_{23}(\delta,\vert S\vert,R,n,\eta,\widehat{d},t,d(x,x'),\lambda_4)\cdot r\cdot
e^{-\lambda_3\cdot r\cdot t}\cdot\lambda_4^{(x\vert x')_{\beta'}}
$$
and
$$
(x\vert x')_{\beta'}\geq (x\vert x')_{x_0}-2Nr-R
$$
with $\lambda_4=max(\lambda,\lambda_1)<1$ and a suitable constant $C_{23}$ 
for every matrix coefficient in (3.14) because the mutual distance of the first vertices of consecutive simplices in (3.14) equals $r$. The number of summands in (3.14) is bounded  
according to (3.12) and (3.13) by $\left(\vert S\vert^{2r}\cdot(1+\vert S\vert)^{(2n+1)R}\right)^N=(C_{24}(\vert S\vert,R,n)\cdot\vert S\vert^r)^{2N}.$
Therefore
$$
\vert\langle(T^*(r)T(r))^N(\alpha),\alpha\rangle\vert\,\leq\,C_{25}^{2N}\cdot
\left(r\cdot (\vert S\vert\cdot e^{-\lambda_3\cdot t}\cdot\lambda_4^{-1})^r\right)^{2N}\cdot\lambda_4^{2N(x\vert x')_{x_0}}
\eqno(3.15)
$$
Suppose now that $N$ is so large that $\lambda_4^{2N}<(1+\vert S\vert)^{-1}$.
As $$\vert\{x_0,\,(x\vert x')_{x_0}=r'\}\vert\,\leq\,(1+d(x,x'))\cdot(1+\vert S\vert)^{r'+3\delta}$$
by (2.5) we deduce from (3.15)
$$
\vert Trace((T(r)^*T(r))^N)\vert\,\leq\,\left(C_{26}\cdot r\cdot(\vert S\vert\cdot e^{-\lambda_3\cdot t}\cdot\lambda_4^{-1})^r\right)^{2N}
\cdot\left(\underset{r'=0}{\overset{\infty}{\sum}}\,\lambda_4^{2Nr'}\cdot(1+\vert S\vert)^{r'}\right)
$$
and
$$
\parallel T(r)\parallel_{\ell^p}\,\leq\,\left( Trace((T^*T)^N)\right)^{\frac1N}
\,\leq\,C_{27}\cdot r\cdot(\vert S\vert\cdot e^{-\lambda_3\cdot t}\cdot\lambda_4^{-1})^r
\eqno(3.16)
$$
for $p\geq 2N$. This shows that the operators $(\Phi^{x,t}-\Phi^{x',t})(r)$ lie in the Schatten class for these values of $p$. For $t>\lambda_3^{-1}\cdot(\log\vert S\vert-\log\lambda_4)$ the series 
$\underset{r}{\sum}\,(\Phi^{x,t}-\Phi^{x',t})(r)$ converges in $\ell^p(\ell^2(\Delta_n^R(\Gamma)),\ell^2(\Delta_{n+1}^R(\Gamma)))$ and its limit coincides with $\Phi^{x,t}-\Phi^{x',t}$. 
\end{proof}

Similarly we get
\begin{lemma}
If $\varrho(\widehat{d})(r)\,=\,O(\lambda^r)$ for some $\lambda<1$, then 
$$
e^{t\cdot\widehat{d}_x}\circ\partial\circ e^{-t\cdot\widehat{d}_x}
\,-\,
e^{t\cdot\widehat{d}_{x'}}\circ\partial\circ e^{-t\cdot\widehat{d}_{x'}}
\,\in\,\ell^p(\ell^2(\Delta_n^R(\Gamma)),\ell^2(\Delta_{n-1}^R(\Gamma)))
$$
and
$$
e^{t\cdot\widehat{d}_x}\circ\pi_{alt}\circ e^{-t\cdot\widehat{d}_x}\,-\,
e^{t\cdot\widehat{d}_{x'}}\circ\pi_{alt}\circ e^{-t\cdot\widehat{d}_{x'}}
\,\in\,\ell^p(\ell^2(\Delta_n^R(\Gamma)))
$$
for all $x,x'\in\Gamma$, $n\geq 0$ and $t>>0,\,p>>0$ large enough.
\end{lemma}

\subsection{The Fredholm module}

For a given integer $R$ let 
$$
{\mathcal H}_*^{R}\,=\,\Lambda^{*+1}(\ell^2(\Gamma))\cap\ell^2(\Delta_*^R(\Gamma))\,\subset\,
\ell^2(\Delta_*(\Gamma)).
\eqno(3.17)
$$
This graded Hilbert space coincides with the image of the antisymmetrization projector $\pi_{alt}$ (2.13) on $\ell^2(\Delta_*^R(\Gamma))$ and is spanned by the canonical orthonormal basis
$$
\{e_{x_0}\wedge e_{x_1}\wedge\ldots\wedge e_{x_n},[x_0,\ldots,x_n]\in\Delta_n^R(\Gamma),\,n\in{\mathbb N}\}.
\eqno(3.18)
$$
We denote by ${\mathcal H}^R_\pm$ the associated ${\mathbb Z}/2{\mathbb Z}$-graded Hilbert space.

\begin{theorem}
Let $(\Gamma,S)$ be a $\delta$-hyperbolic group and let $R\geq 12\delta$ be an integer. Let $d^{MY}$ be the Mineyev-Yu metric (see 3.1) on $\Gamma$. For a given base point $x\in\Gamma$ let $h^x$ be a contracting chain homotopy of the augmented Rips complex $C_*^R(\Gamma,{\mathbb C})$ as constructed in 2.19. For $t>>0$ put, following Lafforgue \cite{La2}, 4.4,
$$
F_{x,t}\,=\,e^{td_x^{MY}}\circ\left(\partial+\pi_{alt}\circ h^x\circ\partial\circ\pi_{alt}\circ h^x \right)\circ e^{-td_x^{MY}}.
\eqno(3.19)
$$
Then 
$$
{\mathcal E}_{R,x,t}\,=\,\left({\mathcal H}_\pm^{R},\,\pi_{reg},\,F_{x,t}\right)
\eqno(3.20)
$$
is a finitely summable weak Fredholm module over $C^*_r(\Gamma)$.
\end{theorem}

For the notion of Fredholm module see section 5.

\begin{proof}
Lemma 3.6 shows that the operators
$$
e^{td_x^{MY}}\circ h^x\circ e^{-td_x^{MY}},\,e^{td_x^{MY}}\circ\partial\circ e^{-td_x^{MY}},\,
e^{td_x^{MY}}\circ\pi_{alt}\circ e^{-td_x^{MY}}
$$
are bounded on $\ell^2(\Delta_*^R(\Gamma))$ in each degree for $t>>0$. As ${\mathcal H}_*^{R}$ vanishes in high degrees because Rips simplices cannot have pairwise different vertices in large dimensions, this implies that $F_{x,t}$ is in fact an odd bounded operator on ${\mathcal H}^R_\pm$ if $t>>0.$ 
Now
$$
\pi(g)\circ\left(e^{td_x^{MY}}\circ h^x\circ e^{-td_x^{MY}}\right)\circ\pi(g)^{-1}\,=\,e^{td_{gx}^{MY}}\circ h^{gx}\circ e^{-td_{gx}^{MY}},
$$
$$
\pi(g)\circ\left(e^{td_x^{MY}}\circ \partial\circ e^{-td_x^{MY}}\right)\circ\pi(g)^{-1}\,=\,e^{td_{gx}^{MY}}\circ\partial\circ e^{-td_{gx}^{MY}},
$$
and
$$
\pi(g)\circ\left(e^{td_x^{MY}}\circ \pi_{alt}\circ e^{-td_x^{MY}}\right)\circ\pi(g)^{-1}\,=\,e^{td_{gx}^{MY}}\circ\pi_{alt}\circ e^{-td_{gx}^{MY}}.
$$
Therefore Proposition 3.7. and Lemma 3.8 imply for $t>>0$ and $p>>0$ large enough that for every $g\in\Gamma$ that the operators
$$
\{[\pi(g),e^{td_x^{MY}}\circ h^x\circ e^{-td_x^{MY}}],\,
[\pi(g),e^{td_x^{MY}}\circ \partial\circ e^{-td_x^{MY}}],\,
[\pi(g),e^{td_x^{MY}}\circ \pi_{alt}\circ e^{-td_x^{MY}}]\}\subset\ell^p(\ell^2(\Delta_*^R(\Gamma)))
$$
are in the Schatten ideal $\ell^p(\ell^2(\Delta_*^R(\Gamma)))$ for $*\geq 0$. It follows that
$$
[\pi(g),F_{x,t}]\in\ell^p({\mathcal H}^R_{\pm}),\,\forall g\in\Gamma.
$$
\\
By construction $h^x$ is a contracting chain homotopy of the augmented Rips complex $C_*^R(\Gamma,{\mathbb C})$. The operator $\pi_{alt}\circ h^x$ is therefore a contracting chain homotopy of the augmented alternating Rips complex $C_*^R(\Gamma,{\mathbb C})_{alt}.$ The map $H^x=\pi_{alt}\circ h^x\circ\partial\circ\pi_{alt}\circ h^x$ is then still a contracting chain homotopy of this complex but satisfies in addition $(H^x)^2=0$. On the non-augmented alternating Rips complex, which is graded by the non-negative integers, this implies that 
$$
(\partial+\pi_{alt}\circ h^x\circ\partial\circ\pi_{alt}\circ h^x)^2\,=\,
\partial\circ H^x+H^x\circ\partial\,=\,Id-H^x_{-1}\circ\partial
$$
and therefore 
$$
Id-F_{x,y}^2\,=\,e^{td_x^{MY}}\circ H^x_{-1}\circ\partial\circ e^{-td_x^{MY}}\,=\,Id-p_x
$$
where 
$$
\begin{array}{cccc}
p_x: & C_0(\Gamma,{\mathbb C}) & \to & C_0(\Gamma,{\mathbb C}),\\
 & & & \\
 & [x_0] & \mapsto & e^{-td^{MY}(x,x_0)}\cdot[x]\\
 \end{array}
$$
is a bounded linear operator of rank one for $t>>0$. This finishes the proof of the theorem.

\end{proof}

\section{The bimodule of Kasparov-Skandalis}

We adopt the construction of a "Gamma"-element by Kasparov and Skandalis \cite{KS}. Following their notation one puts for $T\subset \Gamma$ 
$$
U_T\,=\,\underset{z\in T}{\bigcap}\,B(z,R)\,=\,\{y\in\Gamma,\,diam(T\cup\{y\})\leq R\}.
\eqno(4.1)
$$
Note that $U_T$ is empty if $diam(T)>R$. For $y\in\Gamma$ note $e_y$ the corresponding basis vector of ${\mathbb C}\Gamma$. The following is a slightly modified version of the "radial vector field" introduced in \cite{KS}, Section 7.
 
\begin{definition} 
Let $(\Gamma,S)$ be a $\delta$-hyperbolic group and let $T\subset\Gamma$. 
\begin{itemize}
\item[a)]
If $x\notin U_T$ let
$$
\xi^x_T\,=\,\underset{y\in U_T}{\sum}\left(\underset{z\notin U_T}{\sum}\,c_1^{x,y}(z)\right)e_y
\,\,\,\in {\mathbb C}\Gamma
\eqno(4.2)
$$
where the notations are those of 2.11. 
\item[b)] Define $\zeta^x_T\in {\mathbb C}\Gamma$ by 
$$
\zeta^x_T\,=\,
\begin{cases}
\frac{\xi^x_T}{\parallel\xi^x_T\parallel}, & x\notin U_T,\,\xi^x_T\neq 0,\\
0, & x\notin U_T,\,\xi^x_T=0,\\
e_x, & x\in U_T.\\
\end{cases}
\eqno(4.3)
$$
\end{itemize}
\end{definition}

\begin{prop}
There exists a constant $C_{27}=C_{27}(\delta,\vert S\vert,R,\lambda_1)$ such that 
$$
\parallel \zeta^x_T-\zeta^{x'}_T\parallel_2\,\leq\,C_{27}(\delta,\vert S\vert,R,\lambda_1)\cdot\lambda_1^{(x\vert x')_T}
\eqno(4.4)
$$
for all $x,x'\in\Gamma$ and all $T\subset\Gamma.$
\end{prop}

\begin{proof}
As the vectors $\zeta^x_T$ and $\zeta^{x'}_T$ are of norm zero or one, the assertion holds for 
$(x\vert x')_T\leq R$ if $C\geq 2\lambda_1^{-R}$.
So we may assume $(x\vert x')_T>R$ and have $\{x,x'\}\cap U_T=\emptyset$.
Then
$$
\xi^x_T\,=\,\underset{y\in U_T}{\sum}\left(\underset{z\notin U_T}{\sum}\,c_1^{x,y}(z)\right)e_y
$$
and
$$
\xi^{x'}_T\,=\,\underset{y\in U_T}{\sum}\left(\underset{z\notin U_T}{\sum}\,c_1^{x',y}(z)\right)e_y
$$
so that
$$
\parallel \xi^x_T-\xi^{x'}_T\parallel\,=\,
\parallel\underset{y\in U_T}{\sum}\left(\underset{z\notin U_T}{\sum}\,(c_1^{x,y}(z)-c_1^{x',y}(z))\right)e_y
\parallel 
$$
$$
\leq\,\underset{y\in U_T}{\sum}\,\underset{z\notin U_T}{\sum}\,
\vert c_1^{x,y}(z)-c_1^{x',y}(z)\vert
$$
$$
\leq\,\underset{y\in U_T}{\sum}\,C_1(\delta,\vert S\vert)\cdot\lambda_1^{(x\vert x')_y-10\delta}\,
\leq\,\vert U_T\vert\cdot C_1(\delta,\vert S\vert)\cdot\lambda_1^{(x\vert x')_{U_T}-10\delta}
$$
by (2.22)
$$
\leq (1+\vert S\vert)^R\cdot C_1(\delta,\vert S\vert)\cdot\lambda_1^{(x\vert x')_T-R-10\delta}
\,=\,C_{28}(\delta,\vert S\vert,R,\lambda_1)\cdot\lambda_1^{(x\vert x')_T}.
$$
Let $y\in U_T$ be a point at minimal distance from $x$. As $d(x,z)< d(x,y)$ for all $z\in Supp(\varphi_1(x,y))$ because $x\notin U_T$ by 2.11 we deduce $\underset{z\notin U_T}{\sum}\,c_1^{x,y}(z)=1$ and therefore $\parallel\xi^x_T\parallel_2\geq 1$. Consequently
$$
\parallel \zeta^x_T-\zeta^{x'}_T\parallel_2\,=\,
\parallel\frac{\xi^x_T}{\parallel\xi^x_T\parallel}-\frac{\xi^{x'}_T}{\parallel\xi^{x'}_T\parallel}\parallel\leq\,
2\parallel\xi^x_T-\xi^{x'}_T\parallel\,\leq\,2C_{28}\cdot\lambda_1^{(x\vert x')_T}.
$$
\end{proof}

In the sequel the following fact about hyperbolic spaces will be needed.

\begin{lemma}
Let $(X,d)$ be a $\delta$-hyperbolic geodesic metric space, let $R,r>0$ and let $y,y'\in B(z,R)$ for some $z\in X$. Then 
$$
min(d(u,y),d(u,y'))\geq r\,\,\,\Rightarrow\,\,\,d(u,z)\,\leq\,R-r+2\delta
\eqno(4.5)
$$
for all $u\in geod\{y,y'\}$.
\end{lemma}

\begin{proof}
By hyperbolicity $d(u,geod\{y,z\})\leq\delta$ or $d(u,geod\{y',z\})\leq\delta$. As $y$ and $y'$ play a symmetric role we may suppose that 
the first assertion holds. So let $v\in geod\{y,z\}$ be such that $d(u,v)\leq\delta$. Then
$$
d(u,z)\,\leq\,d(v,z)+d(u,v)\,=\,d(y,z)-d(v,y)+d(u,v)
$$
$$
\leq\,d(y,z)-d(y,u)+2d(u,v)\,\leq\,R-r+2\delta.
$$
\end{proof}

\begin{definition}
For $T\subset\Gamma$ such that $x\notin U_T$ put
$$
V^x_{T}\,=\,\{z\in U_T,\,Supp(\varphi_1(x,z))\nsubseteq U_T\}
\eqno(4.6)
$$
in the notations of 2.11.
\end{definition}

\begin{lemma}
$$
diam(V^x_{T})\leq\, 22\delta.
\eqno(4.7)
$$
\end{lemma}

\begin{proof}
Let $y,y'\in V^x_T$ and suppose that $d(y,y')> 22\delta$. 
Put $u=\overline{yy'}(10\delta)$. Then $d(u,\overline{yx})\leq\delta$ or $d(u,\overline{y'x})\leq\delta$ by hyperbolicity. Suppose that $d(u,\overline{yx})\leq\delta$. 
By assumption there exists $v\in Supp(\varphi_1(x,y)),\,v\notin U_T$. We have $d(u,v)\leq 4\delta$ by (2.21). Lemma 4.3 implies 
$d(z,v)\,\leq\,d(z,u)+d(u,v)\,\leq\,R-10\delta+2\delta+4\delta\leq R$ for all $z\in T$ as $y,y'\in U_T$ and $min(d(u,y),d(u,y')\geq 10\delta$, so that $v\in U_T$. Contradiction !
So $d(u,\overline{y'x})\leq\delta$ and $u'=\overline{y'y}(10\delta)$ satisfies $d(u',\overline{y'x})\leq \delta$ by hyperbolicity because $d(u,u')> 2\delta$. By assumption there exists $v'\in Supp(\varphi_1(x,y')),\,v'\notin U_T$. We have $d(u',v')\leq 4\delta$ as before. Lemma 4.3 leads again to a contradiction because it implies $d(z,v')\leq d(z,u')+d(u',v')\,\leq\,R-10\delta+2\delta+4\delta\leq R$  for all $z\in T$, which is impossible as $v'\notin U_T$.
\end{proof}

\begin{lemma}
Let $T\subset\Gamma$ be such that $x\notin U_T$. Then
$$
\xi^x_T\,=\,\underset{y\in V^x_T}{\sum}\left(\underset{z\notin V^x_T}{\sum}\,c_1^{x,y}(z)\right)e_y
\eqno(4.8)
$$
\end{lemma}

\begin{proof}
If $y\in U_T-V^x_T$, then $\underset{z\notin U_T}{\sum}\,c_1^{x,y}(z)\,=\,0$ by definition of $V^x_T$. Therefore
$$
\xi^x_T\,=\,\underset{y\in U_T}{\sum}\left(\underset{z\notin U_T}{\sum}\,c_1^{x,y}(z)\right)e_y\,
=\,\underset{y\in V^x_T}{\sum}\left(\underset{z\notin U_T}{\sum}\,c_1^{x,y}(z)\right)e_y.
$$
Let now $y,z\in\Gamma$ be such that $y\in V^x_T,\,z\notin V^x_T$ and $c_1^{x,y}(z)\neq 0$. We have to show that $z\notin U_T$. Suppose on the contrary that $z\in U_T$. As $y\in V^x_T$ we may find $z'\in Supp(\varphi_1(x,y)),\,z'\notin U_T$. Moreover $z\in Supp(\varphi_1(x,y))$ because $c_1^{x,y}(z)\neq 0.$ As $z\neq x$ by assumption we have $d(x,y)>10\delta$ and 
$d(\overline{yx}(10\delta),z)\leq 2\delta,$ $d(\overline{yx}(10\delta),z')\leq 2\delta$ and $d(z,z')\leq 4\delta$. Let $v\in Supp(\varphi_1(x,z))$. As $z\notin V^x_T$ one has $v\in U_T$. Therefore $v\neq x$ and $d(\overline{zx}(10\delta),v)\leq 2\delta$. By hyperbolicity, applied to the triangle with vertices
$\overline{yx}(10\delta),z,x$ we may find $w\in\overline{yx}$ such that 
$d(\overline{zx}(10\delta),w)\leq\delta$. Then 
$$
d(\overline{yx}(10\delta),w)\,\geq\,
d(\overline{yx}(10\delta),\overline{zx}(10\delta))-d(w,\overline{zx}(10\delta))\,\geq
$$
$$
\geq\,d(\overline{zx}(10\delta),z)-d(\overline{yx}(10\delta),z)-d(w,\overline{zx}(10\delta))\,\geq\,
10\delta-2\delta-\delta=7\delta
$$
Let $z_i\in T$. By assumption $d(y,z_i)\leq R$ and $d(w,z_i)\leq d(w,v)+d(v,z_i)\leq 3\delta+R$. Lemma 4.3 implies 
$$
d(\overline{yx}(10\delta),z_i)\,\leq\,(R+3\delta)-7\delta+2\delta\leq R-2\delta
$$
so that $d(z',z_i)\leq d(z',\overline{yx}(10\delta))+d(\overline{yx}(10\delta),z_i)\leq 
2\delta+R-2\delta\leq R.$ But this is impossible as $z'\notin U_T$.
\end{proof}

\begin{lemma}
Let $R\geq 33\delta$. Let $T\subset\Gamma$ be such that $x\notin U_T$ and let $y\in V^x_T$. Then
 $x\notin U_{T\cup\{y\}}\cup U_{T\backslash\{y\}}$  and 
$$
V^x_{T}\,=\,V^x_{T\cup\{y\}}\,=\,V^x_{T\backslash\{y\}}.
\eqno(4.9)
$$ 
\end{lemma}

\begin{proof}
Let $v\in V^x_T\subset U_T$. Then $d(v,y)\leq diam(V^x_T)\leq 22\delta\leq R$ by 4.5, so that $v\in U_{T\cup\{y\}}\subset U_T$. On the other hand $Supp(\varphi_1(x,v))\nsubseteqq U_T$ implies $Supp(\varphi_1(x,v))\nsubseteqq U_{T\cup\{y\}}$ so that $v\in V^x_{T\cup\{y\}}$. Thus $V^x_T\subset V^x_{T\cup\{y\}}$. Let now $v'\in V^x_{T\cup\{y\}}$. Then $v'\in U_{T\cup\{y\}}\subset U_T$ and $d(y,v')\,\leq\,diam(V^x_{T\cup\{y\}})\,\leq\,22\delta$ because $y\in V^x_T\subset V^x_{T\cup\{y\}}$. But then $Supp(\varphi_1(x,v'))\subset B(v',10\delta)\subset B(y,33\delta)\subset B(y,R),$ so that $Supp(\varphi_1(x,v')\nsubseteqq U_{T\cup\{y\}}$ implies $Supp(\varphi_1(x,v')\nsubseteqq U_T$ and $v'\in V^x_T$. This proves the first equality.\\
We want to apply it now to $T\backslash\{y\},\,y\in T\cap V^x_T\subset U_{T\backslash\{y\}}$. For this we have to show that $x\notin U_{T\backslash\{y\}}$ and $y\in V^x_{T\backslash\{y\}}$. Suppose that $x\in U_{T\backslash\{y\}}$. Then $d(x,z)\leq R$ and $d(y,z)\leq R$ for $z\in T\backslash\{y\}$, but $d(x,y)>R\geq 33\delta$ because $x\notin U_T$. Applying Lemma 4.3 to the point $\overline{yx}(10\delta)$ and using (2.21) as before shows that $Supp(\varphi_1(x,y))\subset U_T$ which would contradict $y\in V^x_T$. So $x\notin U_{T\backslash\{y\}}$.
By definition $Supp(\varphi_1(x,y))\subset B(y,10\delta)\subset B(y,R)$ and as $y\in V^x_T$ we may find $v\in Supp(\varphi_1(x,y))$ such that $v\notin U_T$. But then $d(v,z)>R$ for some $z\in T$, which is necessarily different from $y$, so that $z\in T\backslash\{y\}$ and $y\in V^x_{T\backslash\{y\}}$. So we may deduce from the first part of the lemma that $V^x_{T\backslash\{y\}}=V^x_T$.
\end{proof}
\\
\\
Recall that exterior multiplication with $\xi\in{\mathbb C}\Gamma$ defines a bounded operator
$$
\mu(\xi):\Lambda^*(\ell^2(\Gamma))\to \Lambda^{*+1}(\ell^2(\Gamma)),
\,\omega\mapsto\xi\wedge\omega.
\eqno(4.10)
$$
"Clifford multiplication" by $\xi$ is given by the selfadjoint odd bounded operator
$$
cl(\xi)\,=\,\mu(\xi)+\mu(\xi)^*:\Lambda^*(\ell^2(\Gamma))\to \Lambda^{*\pm 1}(\ell^2(\Gamma)).
\eqno(4.11)
$$
It satisfies the identity
$$
cl(\xi)^2\,=\,\parallel\xi\parallel^2\, Id.
\eqno(4.12)
$$

The main result of this section is 
\begin{theorem}
Let $(\Gamma,S)$ be a $\delta$-hyperbolic group and let $R\geq 48\delta$.
Let $x\in\Gamma$ be a base vertex of the Rips complex $\Delta^R_*(\Gamma)$.
\begin{itemize}
\item[a)] The linear map
$$
\begin{array}{cccc}
F_{x}: & \Lambda^*(\ell^2(\Gamma)) & \longrightarrow & \Lambda^{*\pm 1}(\ell^2(\Gamma)) \\
 & & & \\
& e_{x_0}\wedge e_{x_1}\wedge\ldots\wedge e_{x_n} & \mapsto & cl(\zeta^x_{\{x_0,\ldots,x_n\}})(e_{x_0}\wedge e_{x_1}\wedge\ldots\wedge e_{x_n})\\
\end{array}
\eqno(4.13)
$$
defines an odd bounded, selfadjoint linear operator on ${\mathcal H}^R_\pm$ (see section 3).
\item[b)]
The triple
$$
{\mathcal E}_{x,R}\,=\,({\mathcal H}^R_\pm,\,\pi_{reg},\,F_{x})
\eqno(4.14)
$$
defines a finitely summable Fredholm module over $C^*_r(\Gamma)$.
\item[c)] This Fredholm module is $p$-summable for 
$$
p\,>\,20\delta\cdot\log(1+\vert S\vert)\cdot(1+\vert S\vert)^{2\delta}
\eqno(4.15)
$$
\end{itemize}
\end{theorem}

\begin{proof}
Let 
$$
{\mathcal H}^x_0\,=\,Vect\{e_\alpha,\,\alpha\in\Delta^R(\Gamma),\,x\in U_{Supp(\alpha)}\}
$$ 
and put for every subset $W\subset\Gamma$ of diameter at most $22\delta$ 
$$
{\mathcal H}^x_W\,=\,Vect\{e_\beta,\,\beta\in\Delta^R(\Gamma),\,x\notin U_{Supp(\beta)},\,V^x_{Supp(\beta)}=W\}.
$$
Then ${\mathcal H}^R$ is the Hilbert sum of the finite dimensional subspaces ${\mathcal H}^x_0$ and ${\mathcal H}^x_W,\,W\subset\Gamma.$
Lemma 4.6 and 4.7 show that these subspaces are invariant under $F_{x}$ and that its restriction to each of these subspaces is given 
by Clifford multiplication with a real unit vector $\zeta_W\in\ell^2(W)$ or with $e_x$. Therefore $F_{x}$ is a bounded selfadjoint linear operator of norm one 
which satisfies $F_{x}^2=Id$ in strictly positive degrees. To understand the operator in degree 0 we adjoin a copy of $\mathbb C$ in degree -1 and may thus calculate in the full Clifford module $\Lambda^*(\ell^2(\Gamma))\cap\ell^2(\Delta_{*-1}^R(\Gamma)),\,*\geq 0.$ For a Rips 0-simplex $[x_0]$  
$U_{x_0}=B(x_0,R)$. So, if $d(x,x_0)>R$ one has $x_0\notin V^x_{\{x_0\}}$ and $id=cl(\zeta^x_{\{x_0\}})^2=\mu(\zeta^x_{\{x_0\}})^*\circ\mu(\zeta^x_{\{x_0\}})=
F_{x}^2$ on ${\mathcal H}^x_W,\,W=V^x_{\{x_0\}}$ and $F_{x}^2(e_{x_0})=e_{x_0}$. If $d(x,x_0)\leq R,\,x\neq x_0,$ then $e_{x_0}\in{\mathcal H}^x_0$ and 
$F_{x}^2([x_0])=\mu(e_x)^*\circ\mu(e_x)(e_{x_0})=e_{x_0}$. Finally $F_{x}(e_x)=0$. Thus 
$$
F_{x}^2=1-\pi_x
\eqno(4.16)
$$
where $\pi_x$ is the orthogonal projection onto the subspace spanned by $e_x$.\\
\\
Proposition 4.2 and the argument used in the proof of Proposition 3.7 permits to conclude that the operators $[F_{x},\pi(g)]\,=\,(F_{x}-F_{gx})\circ\pi(g^{-1}),\,g\in\Gamma$ are finitely summable. A closer look at the proof of 3.7 allows to deduce from 4.2 that (4.14) is $p$-summable for 
every $p$ satisfying the inequality 
$$
\lambda_1^p<(1+\vert S\vert)^{-1}
\eqno(4.17)
$$
According to Mineyev \cite{M}, pp. 815-816 
$$
\lambda_1\,=\,\left(1-\frac{1}{(1+\vert S\vert)^{2\delta}}\right)^{\frac{1}{10\delta}}
\eqno(4.18)
$$
is an admissible choice for his constant (note that it depends only on $\delta$ and $\vert S\vert$ as claimed in (2.22)). One has
$$
\lambda_1^{-10\delta}\,=\,\left(1-\frac{1}{(1+\vert S\vert)^{2\delta}}\right)^{-1}\,>\,
1+\frac{1}{(1+\vert S\vert)^{2\delta}}
$$
and therefore
$$
\log(\lambda_1^{-10\delta})\,>\,\log(1+(1+\vert S\vert)^{-2\delta})\,>\,
(1+\vert S\vert)^{-2\delta}-\frac12(1+\vert S\vert)^{-4\delta}\,>\,\frac12(1+\vert S\vert)^{-2\delta}.
$$
So if $$p\,>\,20\delta\cdot\log(1+\vert S\vert)\cdot(1+\vert S\vert)^{2\delta}$$ as proposed in (4.15) one gets
$$
\log(\lambda_1^{-p})\,=\,\frac{p}{10\delta}\cdot\log(\lambda_1^{-10\delta})\,>\,
2\log(1+\vert S\vert)\cdot(1+\vert S\vert)^{2\delta}\cdot\frac12(1+\vert S\vert)^{-2\delta}
\,=\,\log(1+\vert S\vert)
$$
or
$$
\lambda_1^p\,<\,(1+\vert S\vert)^{-1}
$$
as desired.
This finishes the proof of the theorem.
\end{proof}

\section{The Gamma element}

We recall a few facts about Kasparov's bivariant $K$-theory \cite{Ka}. Let $G$ be a locally compact second countable group. There is a universal stable and split-exact homotopy bifunctor
$$
KK^G:\,\text{$G$-$C^*$-Alg}\,\times\,\text{$G$-$C^*$-Alg}\,\to\,\text{Ab}
\eqno(5.1)
$$
from the category of separable $G$-$C^*$-algebras to the category of abelian groups. It comes equipped with a bilinear and associative product
$$
KK^G(A,B)\,\otimes_{\mathbb Z}\,KK^G(B,C)\,\to\,KK^G(A,C),\,\,\,\forall A,B,C\in \text{$G$-$C^*$-Alg}.
\eqno(5.2)
$$
The product turns the groups $KK^G(A,A)$ into unital associative rings. Equivariant $KK$-theory generalizes Kasparov's bivariant $K$-theory $KK_*(-,-)$ of $C^*$-algebras which corresponds to the case $G=1$. The universal property implies that every homomorphism $H\to G$ of locally compact groups gives rise to a natural transformation 
$$
res^G_H:\,KK^G(A,B)\to KK^H(A,B),
\eqno(5.3)
$$
as well as to natural transformations
$$
j:\,KK^G(A,B)\,\to\,KK(A\rtimes G,B\rtimes G)
\eqno(5.4)
$$
and
$$
j_r:\,KK^G(A,B)\,\to\,KK(A\rtimes_r G,B\rtimes_r G)
\eqno(5.5)
$$
from equivariant bivariant $K$-theory to the $K$-theory of the full and the reduced crossed products, respectively. All these transformations preserve Kasparov products. The full and the reduced crossed product coincide for proper $G$-$C^*$-algebras.\\
\\
A bivariant $K$-theory class $\gamma\in KK^G({\mathbb C},{\mathbb C})$ is a {\bf "Gamma"-element} 
\cite{Ka}, \cite{Tu} for $G$ if it is in the image of the Kasparov product
$$
KK^G({\mathbb C},A)\,\otimes_{\mathbb Z}\,KK^G(A,{\mathbb C})\,\to\,KK^G({\mathbb C},{\mathbb C})
\eqno(5.6)
$$
for a proper $G$-$C^*$-algebra $A$ and satisfies 
$$
res^G_K(\gamma)\,=\,1\in KK^K({\mathbb C},{\mathbb C})
\eqno(5.7)
$$ 
for all compact subgroups $K\subset G$. This implies 
$$
\alpha\circ\beta\,=\,1\,\in KK^G(A,A)
\eqno(5.8)
$$
for every factorization $\gamma=\beta\circ\alpha,\,\alpha\in KK^G(A,{\mathbb C}),\,\beta\in KK^G({\mathbb C},A)$ with $A$ proper. A "Gamma"-element is unique if it exists \cite{Tu}.\\
\\
For $G=\Gamma$ a discrete group and a $\Gamma$-$C^*$-algebra $A$ there exists a tautological isomorphism
$$
\iota:\,KK^\Gamma(A,{\mathbb C})\,\simeq\,KK(A\rtimes\Gamma,{\mathbb C}) 
\eqno(5.9)
$$
\\
between the equivariant $K$-homology of $A$ and the $K$-homology of the universal (or full) crossed product $C^*$-algebra $A\rtimes\Gamma$. It equals the composition
$$
\iota:\,KK^\Gamma(A,{\mathbb C})\,\overset{j}{\longrightarrow}\,KK(A\rtimes\Gamma,C^*\Gamma)\,\overset{\pi_*}{\longrightarrow}\,KK(A\rtimes\Gamma,{\mathbb C})
\eqno(5.10)
$$
where $\pi:C^*\Gamma\to{\mathbb C}$ is the trivial representation. In particular the diagram
$$
\begin{array}{ccccc}
& KK^\Gamma(A',B')\otimes_{\mathbb Z} KK^\Gamma(B',{\mathbb C}) & \overset{\circ}{\longrightarrow} & 
KK^\Gamma(A',{\mathbb C}) & \\
&&&& \\
j\otimes\iota & \downarrow &&  \downarrow & \iota \\
&&&&\\
& KK(A'\rtimes\Gamma,B'\rtimes\Gamma)\otimes_{\mathbb Z} KK(B'\rtimes\Gamma,{\mathbb C}) & \overset{\circ}{\longrightarrow} & 
KK(A'\rtimes\Gamma,{\mathbb C}) & \\
\end{array}
\eqno(5.11)
$$
commutes for all $\Gamma$-$C^*$-algebras $A',B'$. 

\begin{prop}
Let $\Gamma$ be a countable discrete group and suppose that a "Gamma"-element $\gamma\in KK^\Gamma({\mathbb C},{\mathbb C})$ exists. Then there is a unique class $\gamma_r\,\in\,KK(C^*_r(\Gamma),{\mathbb C})$
such that 
$$
\gamma_r\,=\,j_r(\beta)\circ \iota(\alpha)\,\in KK(C^*_r\Gamma,{\mathbb C})
\eqno(5.12)
$$
for any factorization $\gamma\,=\,\beta\circ\alpha,\,\alpha\in KK^\Gamma(A,{\mathbb C}),\,\beta\in KK^\Gamma({\mathbb C},A)$ of $\gamma$ with $A$ proper. It satisfies 
$$
p_*\circ\gamma_r\,=\,\iota(\gamma)\,\in\,KK(C^*\Gamma,{\mathbb C})
\eqno(5.13)
$$
where $p:\,C^*\Gamma\to C^*_r(\Gamma)$ is the canonical epimorphism. The class $\gamma_r$ is called the {\bf reduced "Gamma"-element"} of $\Gamma$.
\end{prop}

\begin{proof}
It only has to be shown that the class on the r.h.s. of (5.12) is independent of the factorization of $\gamma$. So let 
$\gamma=\beta_A\circ\alpha_A=\beta_B\circ\alpha_B,\,
\alpha_A\in KK^\Gamma(A,{\mathbb C}),\,\beta_A\in KK^\Gamma({\mathbb C},A),\,\alpha_B\in KK^\Gamma(B,{\mathbb C}),\,\beta_B\in KK^\Gamma({\mathbb C},B)$ be two factorizations of $\gamma$ with $A$ and $B$ proper. The associativity of the Kasparov product and the uniqueness of the "Gamma"-element imply
$$
\alpha_A\,=\,1\circ\alpha_A\,=\,(\alpha_A\circ\beta_A)\circ\alpha_A\,=\,\alpha_A\circ(\beta_A\circ\alpha_A)\,=\,\alpha_A\circ(\beta_B\circ\alpha_B)\in KK^\Gamma(A,{\mathbb C}),
$$ 
and
$$
\beta_B\,=\,\beta_B\circ 1\,=\,\beta_B\circ(\alpha_B\circ\beta_B)\,=\,
(\beta_B\circ\alpha_B)\circ\beta_B\,=\,
(\beta_A\circ\alpha_A)\circ\beta_B\in KK^\Gamma({\mathbb C},B).
$$
Moreover 
$$
j(\alpha_A\circ\beta_B)=j_r(\alpha_A\circ\beta_B)\,\in KK(A\rtimes\Gamma,B\rtimes\Gamma)\,=\,
KK(A\rtimes_r\Gamma,B\rtimes_r\Gamma)
$$ because $A$ and $B$ are proper so that 
$$
j_r(\beta_A)\circ\iota(\alpha_A)\,=\,
j_r(\beta_A)\circ\iota((\alpha_A\circ\beta_B)\circ\alpha_B)\,=\,
j_r(\beta_A)\circ j(\alpha_A\circ\beta_B)\circ\iota(\alpha_B)
$$

$$
=\,
j_r(\beta_A)\circ j_r(\alpha_A\circ\beta_B)\circ\iota(\alpha_B)\,=\,
j_r(\beta_A\circ\alpha_A\circ\beta_B)\circ\iota(\alpha_B)\,=\,
j_r(\beta_B)\circ\iota(\alpha_B).
$$
Concerning the claim (5.13) we note that 
$$
p_*\circ j_r(\beta_A)\,=\,j(\beta_A)
$$
because $A$ is proper, so that
$$
p_*\circ\gamma_r\,=\,p_*\circ(j_r(\beta_A)\circ\iota(\alpha_A))\,=\,
j(\beta_A)\circ\iota(\alpha_A)\,=\,j(\beta_A)\circ\iota(\alpha_A\circ 1)
$$
$$
=\,(j(\beta_A)\circ j(\alpha_A))\circ\iota(1)
=\,j(\beta_A\circ\alpha_A)\circ\iota(1)
$$
$$
=\,j(\gamma)\circ\iota(1)\,=\,\iota(\gamma).
$$
\end{proof}

Kasparov's bivariant $K$-theory is realized as the group of homotopy classes of Kasparov-bimodules (with addition induced by the direct sum of bimodules). For our needs it suffices to give a description of Kasparov $(A,B)$-bimodules in the case $B={\mathbb C}$, the $C^*$-algebra of complex numbers.
These are called Fredholm modules.\\
\\
An even {\bf Fredholm-module} over a unital $C^*$-algebra 
$A$ is a triple 
$$
{\mathcal E}\,=\,({\mathcal H}_{\pm},\,\varrho,\,F),
\eqno(5.14)
$$
where ${\mathcal H}_{\pm}$ is a ${\mathbb Z}/2{\mathbb Z}$-graded complex Hilbert space, $\varrho:A\to{\mathcal L}({\mathcal H})_+$ is an even non-degenerate representation of $A$ on ${\mathcal H}_\pm$, and $F\in{\mathcal L}({\mathcal H})_-$ is an odd, bounded linear operator satisfying
$$
F^2-id\,\in\,{\mathcal K}({\mathcal H}),
\eqno(5.15)
$$
$$
F-F^*\,\in\,{\mathcal K}({\mathcal H}),
\eqno(5.16)
$$
and
$$
[F,\varrho(a)]\in{\mathcal K}({\mathcal H})
,\,\,\,\forall a\in A.
\eqno(5.17)
$$
A {\bf weak Fredholm module} is a triple as above satisfying only conditions (5.15) and (5.17), but not necessarily (5.16). As every weak Fredholm module is canonically homotopic to a genuine Fredholm module, Kasparov's $K$-homology groups may as well be defined as 
the group of homotopy classes of weak Fredholm modules \cite{Bl} (see also Lemma 5.3).\\
\\
Following \cite{Co}, Appendix 2, and \cite{EN},2.2, we call a weak Fredholm module (Fredholm module) over $A$ {\bf $p$-summable} over the dense subalgebra ${\mathcal A}\subset A$ if 
$$
F^2-id\,\in\,\ell^p({\mathcal H}),
\,\,\,[F,\varrho(a')]\in\ell^p({\mathcal H}),\,\,\,\text{(and}\,\,\,F-F^*\in\ell^p({\mathcal H}))
\eqno(5.18)
$$
for all $a'\in{\mathcal A}$. Here 
$\ell^p({\mathcal H})\subset{\mathcal K}({\mathcal H})$ denotes the Schatten ideal of compact operators in ${\mathcal H}$ with $p$-summable sequence of singular values. It is called {\bf finitely summable} over ${\mathcal A}$ if it is $p$-summable for $p>>0$.\\
\\
An {\bf operator homotopy} between $p$-summable (weak) Fredholm modules over $(A,{\mathcal A})$ is a family 
${\mathcal E}_t=({\mathcal H}_{\pm},\,\varrho,\,F_t),\,\,t\in[0,1],$ of (weak) Fredholm modules over $A$, which are $p$-summable over ${\mathcal A}$ and such that $t\mapsto F_t\in{\mathcal L}({\mathcal H})$ is continuous in the strong $*$-topology. \\
\\
Finitely summable Fredholm modules possess nice regularity properties. In particular, the Chern character of a finitely summable Fredholm module in cyclic cohomology can be given by a simple formula. It is therefore an interesting question whether a given $K$-homology class 
can be realized by a finitely summable Fredholm module. We are going to answer this question affirmatively for the reduced $\gamma$-element of a word-hyperbolic group. 
This settles a problem posed \cite{EN}, section 1.\\
\\
\begin{theorem}
Let $(\Gamma,S)$ be a $\delta$-hyperbolic group, where $\delta>0$ is supposed to be an even integer. Let $R\geq 48\delta$ . 
The modified Lafforgue-bimodules of Theorem 3.9 and the modified Kasparov-Skandalis-bimodules
of Theorem 4.8 represent the reduced "Gamma"-element of $\Gamma$.
\end{theorem}

\begin{proof}
For the proof of the first assertion we adopt the notations of section 3. For $0\leq s\leq 1$ let 
$$
\widehat{d}(s)\,=\,(1-s)\cdot d_{MY}\,+\,s\cdot d_{Laff}:\,\Gamma\times\Gamma\,\to\,{\mathbb R}_+
\eqno(5.19)
$$
be a convex combination of the Mineyev-Yu and the Lafforgue metric on $\Gamma$. The proof of 3.9 and Lemma 3.6 show that 
$$
{\mathcal E}_{R,x,t}(s)\,=\,\left({\mathcal H}_\pm^{R},\,\pi_{reg},\,e^{t\widehat{d}_x(s)}\circ\left(\partial+\pi_{alt}\circ h^x\circ\partial\circ\pi_{alt}\circ h^x \right)\circ e^{-t\widehat{d}_x(s)}\right)
\eqno(5.20)
$$
is a family of weak Fredholm modules over $C^*_r(\Gamma)$. As $s\mapsto F_{x,t}(s)$ is strongly $*$-continuous we learn that the weak Fredholm modules 
${\mathcal E}_{R,x,t}={\mathcal E}_{R,x,t}(0)$ and 
$$
{\mathcal E}_{R,x,t}(1)\,=\,\left({\mathcal H}_\pm^{R},\,\pi_{reg},\,e^{td^{Laff}_x}\circ\left(\partial+\pi_{alt}\circ h^x\circ\partial\circ\pi_{alt}\circ h^x \right)\circ e^{-t d^{Laff}_x}\right)
$$ 
are operator homotopic for $t>>0$ sufficiently large. \\
Let now
$$
h^x(s')\,=\,(1-s')\pi_{alt}\circ h^x\,+\,s' h^x_{Laff},\,\,s'\in[0,1].
\eqno(5.21)
$$
As convex combinations of contracting chain homotopies these are again contracting chain homotopies of the augmented alternating Rips complex t for $R>>0$ sufficiently large. Moreover 
$$
{\mathcal E}'_{R,x,t}(s')\,=\,\left({\mathcal H}_\pm^{R},\,\pi_{reg},\,e^{td^{Laff}_x}\circ\left(\partial+h^x(s')\circ\partial\circ h^x(s')\right)\circ e^{-td^{Laff}_x}\right)
\eqno(5.22)
$$
is a family of weak Fredholm modules over $C^*_r(\Gamma)$. As $s'\mapsto F'_{R,x,t}(s')$ is continuous in operator norm we learn that the weak Fredholm modules 
${\mathcal E}_{R,x,t}(1)={\mathcal E}'_{R,x,t}(0)$ and ${\mathcal E}'_{R,x,t}(1)={\mathcal E}^{Laff}_{R,x,t}$ are operator homotopic for $R$ and $t>>0$ large enough. The Lafforgue-bimodule ${\mathcal E}^{Laff}_{R,x,t}$ represents the reduced "Gamma"-element by \cite{La2}, Section 5, so that the same is true for our bimodule ${\mathcal E}_{R,x,t}$ of 3.9.
\\
\\
For the proof of the second assertion we adopt the notations of section 4. Let $R\geq 48\delta$ and put $k=3\delta$. With these choices the conditions (C1), (C2) and (C3) of \cite{KS} , pages 187 and 190 are satisfied. Fix $W\subset \Gamma$ such that $d(x,W)>R$ and let $T\subset\Gamma$ satisfy 
$V^x_T=W$. It follows from (2.21) that 
$$
\{y\in U_T,d(x,y)\leq d(x,U_T)+8\delta\}\subset V^x_T=W.
\eqno(5.23)
$$
This implies

\begin{itemize}
\item 
$$
r_{T,x}\,=\,\Sup\,I(d(x,W))-3\delta\,=\,r'_{W,x}\,\leq\, d(x,W)-30\delta
\eqno(5.24)
$$
(see \cite{KS}, page 191, line 14),\\
\item 
$$
Y_{T,x,r}\,=\,\underset{y\in B(x,r)}{\bigcup}\,\{a\in W,\,d(x,a)\leq d(x,W)+2\delta\}\,=\,Y'_{W,x,r}
\eqno(5.25)
$$
for $0\leq r\leq r'_{W,x}$ (see \cite{KS}, page 189, lines 14 and 17),\\
\item 
$$
\widetilde{\psi}_{T,x}\,=\,\left(f'_{W,x,0}+\underset{0}{\overset{r'_{W,x}}{\int}}\,f'_{W,x,t}dt\right)\cdot\mu_W\,=\,\widetilde{\psi}'_{W,x}
\eqno(5.26)
$$
where $f'_{W,x,r}=\chi_{Y'_{W,x,r}}$ (see \cite{KS}, page 191, line 27),
\end{itemize}

so that the operator $F_x^{KS}$ is given on the subspace ${\mathcal H}^x_W$ by Clifford multiplication with the unit vector 
$\phi'_{x,W}$ attached to $\widetilde{\psi}'_{W,x}$ as in \cite{KS}, page 191, line 29 and page 192, line 34. Because 
$Supp(\phi'_{x,W})\subset W$ the subspace ${\mathcal H}_x^W$ is invariant under $F_x^{KS}$.\\
Let now
$$
{\mathcal H}'\,=\,\underset{W,d(x,W)>R}{\bigoplus}\,{\mathcal H}^x_W.
\eqno(5.27)
$$
This is a closed subspace of ${\mathcal H}^R$ of finite codimension which is invariant under $F_x$ and $F_x^{KS}$. For $0\leq t\leq 1$ let $F_x(t)\in{\mathcal L}({\mathcal H}^R)$ 
be the operator which vanishes on $({\mathcal H}')^\perp$ and satisfies 
$$
F(t)\vert_{{\mathcal H}^x_W}\,=\,cl(\zeta^x_W(t)),\,\,\,\zeta^x_W(t)=\frac{\xi^x_W(t)}{\parallel\xi^x_W(t)\parallel},\,\,\,\xi^x_W(t)\,=\,(1-t)\cdot\zeta^x_W+t
\cdot\widetilde{\phi}'_{x,W}
\eqno(5.28)
$$
if $d(x,W)>R$. This is a well defined operator because both vectors $\zeta^x_W$ and $\widetilde{\phi}'_{x,W}$ are positive linear combinations of points of $W$ so that no convex combination of them vanishes. In fact
$$
\parallel\xi^x_W(t)\parallel_2\,\geq\,(dim \,{\mathcal H}^x_W)^{-\frac12}\parallel\xi^x_W(t)\parallel_1\,=\,
(dim {\mathcal H}^x_W)^{-\frac12}\cdot\left( (1-t)\parallel\zeta^x_W\parallel_1+t
\parallel\widetilde{\phi}'_{x,W}\parallel_1\right)
$$
$$
\geq\,(dim \,{\mathcal H}^x_W)^{-\frac12}\,\left( (1-t)\parallel\zeta^x_W\parallel_2+t
\parallel\widetilde{\phi}'_{x,W}\parallel_2\right)\,=\,(dim\,{\mathcal H}^x_W)^{-\frac12}\,\geq\,C_{29}(\vert S\vert,R)\,>\,0
$$
which shows that $t\mapsto F_x(t)$ is continuous with respect to the operator norm. The same estimate guarantees also that 
$[F_x(t),\pi_{reg}(\Gamma)]\subset{\mathcal K}({\mathcal H}^R)$ because 
$$
\underset{d(x,W)\to\infty}{\lim}\,(\zeta^x_W-\zeta^{gx}_W)\,=\,\underset{d(x,W)\to\infty}{\lim}\,(\widetilde{\phi}'_{x,W}-\widetilde{\phi}'_{gx,W})\,=\,0
$$
for all $g\in\Gamma$ by 4.2 and \cite{KS}, 6.9. This shows that our Fredholm module ${\mathcal E}_{x,R}$ of 4.8 is operator homotopic to the Kasparov-Skandalis bimodule \cite{KS}, p.192, 6.10, which represents the reduced "Gamma"-element.
\end{proof}
\\
\\
There is still a little difference between the two cases considered in the previous theorem: whereas the Kasparov-Skandalis method yields genuine Kasparov bimodules Lafforgue's approach only leads to weak ones. This ambiguity can actually be ignored because of the following result.\\
\\
Recall \cite{Pu}, 2.3, 2.4, that for $p\geq 1$ the $p$-summable smooth $K$-homology groups 
${\mathcal K}{\mathcal K}^{(p)}((A,{\mathcal A}),{\mathbb C})$ of the separable $C^*$-algebra $A$
with respect to the dense involutive subalgebra $\mathcal A$ are defined as the abelian group of equivalence classes of Fredholm modules over $A$ which are $p$-summable over $\mathcal A$ with respect to the equivalence relation generated by unitary equivalence, addition of degenerate modules (i.e. modules for which the expressions (5.15), (5.16) and (5.17) are identically zero), and smooth operator homotopy. Denote by ${\mathcal K}{\mathcal K}^{(p)}_{weak}((A,{\mathcal A}),{\mathbb C})$ the corresponding group of equivalence classes of weak Fredholm modules. Then we have

\begin{lemma}
Let $A$ be a separable $C^*$-algebra and let $\mathcal A$ be an involutive dense subalgebra. Then the forgetful map 
$$
{\mathcal K}{\mathcal K}^{(p)}((A,{\mathcal A}),{\mathbb C})\,\to\,{\mathcal K}{\mathcal K}^{(p)}_{weak}((A,{\mathcal A}),{\mathbb C})
\eqno(5.29)
$$
is an isomorphism of abelian groups.
\end{lemma}

\begin{proof}
Let $F\in {\mathcal L}({\mathcal H})$ be a bounded linear operator on a Hilbert space. Then\\ $(F-F^*)(F-F^*)^*\,=\,(F-F^*)(F^*-F)$ is a positive operator, so that\\ $1\,+\,\frac14(F-F^*)(F^*-F)$ is invertible:
$$
T\,=\,(1\,+\,\frac14(F-F^*)(F^*-F))^{-1}\,\in\,{\mathcal L}({\mathcal H}).
\eqno(5.30)
$$
\\
\\
Let $({\mathcal R},*)$ be an abstract unital involutive algebra and let $F\in{\mathcal R}$ satisfy $F^2=1$. Suppose that $1+\frac14(F-F^*)(F^*-F)$ is invertible with inverse $T\in{\mathcal R}$. Then 
$$
\widetilde{F}\,=\,\left(\frac14(FF^*-F^*F)\,+\,\frac12(F+F^*)\right)T\,\in\,{\mathcal R}
\eqno(5.31)
$$
satisfies 
$$
\widetilde{F}^2\,=\,1\,\,\,\text{and}\,\,\,\widetilde{F}^*\,=\,\widetilde{F}
$$
and
$$
F_t\,=\,\left(1\,+\,\frac{t}{2}(F-\widetilde{F})\right)F\left(1\,+\,\frac{t}{2}(\widetilde{F}-F)\right),\,\,\,t\in {\mathbb R}
\eqno(5.32)
$$
is a one parameter family of elements satisfying $F_t^2=id$ such that 
$F_0=F$ and $F_1=\widetilde{F}$. The family is constant if $F$ itself is selfadjoint. This is Lemma 4.6.2 of \cite{Bl}, where we have used the canonical bijection $e\mapsto F=2e-1$ between the set of idempotents and of elements of square one in a unital algebra.
\\
\\
Let ${\mathcal E}\,=\,({\mathcal H},\varrho,F)$ be a weak Fredholm module over the $C^*$-algebra $B$. Let $\pi:{\mathcal L}({\mathcal H})\to{\mathcal L}({\mathcal H})/{\mathcal K}({\mathcal H})$ be the quotient homomorphism. Then $\pi(F)$ and $\pi(F^*)$ commute with $\pi\circ\varrho(A)$. By step 1 the expressions (5.31) and (5.32) make sense in ${\mathcal L}({\mathcal H})$. Moreover $\pi(F_t)$ commutes with $\pi\circ\varrho(A)$ for all $t\in{\mathbb R}$, equals $\pi(F)$ for $t=0$ and is selfadjoint for $t=1$. Thus ${\mathcal E}_t\,=\,({\mathcal H},\varrho,F_t),\,t\in[0,1]$ defines an operator homotopy between $\mathcal E$ and a genuine Fredholm module. This construction is invariant under unitary equivalence, presrves operator homotopies and sends degenerate weak Fredholm modules to degenerate Fredholm modules.
It therefore descends to equivalence classes and shows that the forgetful map from the set of homotopy classes of genuine Fredholm modules to the set of homotopy classes of weak Fredholm modules over $A$ is a bijection.
\\
\\
If ${\mathcal E}\,=\,({\mathcal H},\varrho,F)$ is a weak Fredholm module over $A$ which is $p$-summable over the dense involutive subalgebra ${\mathcal A}\subset A$, we may repeat the previous reasoning with $\varrho$ replaced by its restriction $\varrho'$ to 
$\mathcal A$ and $\pi$ replaced by the quotient homomorphism 
$\pi': {\mathcal L}({\mathcal H})\to{\mathcal L}({\mathcal H})/\ell^p({\mathcal H})$ and obtain thus our claim.
\end{proof}
\\
\\
Altogether we have shown

\begin{theorem}
Let $(\Gamma,S)$ be a $\delta$-hyperbolic group. Then the reduced Gamma-element $\gamma_r\in KK(C^*_r\Gamma,{\mathbb C})$ may be 
represented by a Fredholm module which is $p$-summable over ${\mathbb C}\Gamma$ for 
$$
p\,>\,20\delta\cdot\log(1+\vert S\vert)\cdot(1+\vert S\vert)^{2\delta}.
\eqno(5.33)
$$
\end{theorem}

{\vskip8mm}
.\\
I2M, UMR 7373 du CNRS, Campus de Luminy,\\
Universit\'e d'Aix-Marseille, France\\
\\
\\
Email: jmcabrera@laposte.net, michael.puschnigg@univ-amu.fr

\end{document}